\documentclass[11pt]{article}
\usepackage[usenames,dvipsnames]{xcolor}
\usepackage{booktabs}
\usepackage[sc,noBBpl]{mathpazo}
\usepackage[mathscr]{eucal}
\usepackage{amsmath,amsfonts,amssymb,amsthm,amscd}
\usepackage{mathtools}
\usepackage[square,numbers]{natbib}
\usepackage{multirow}
\theoremstyle{plain}
\newtheorem{theorem}{Theorem}
\newtheorem{lemma}[theorem]{Lemma}
\newtheorem{proposition}[theorem]{Proposition}

\newtheorem{definition}[theorem]{Definition}
\newtheoremstyle{note}{\topsep}{\topsep}{\slshape}{}{\scshape}{}{ }{}
\theoremstyle{note}

\newtheorem{remark}[theorem]{Remark}
\newtheoremstyle{warning}{\topsep}{\topsep}{\color{Blue}\small\slshape}{}{\color{Red}\scshape}{.}{ }{}
\theoremstyle{warning}

\numberwithin{equation}{section}
\numberwithin{theorem}{section}

\usepackage[sc,small,compact]{titlesec}
\newcommand\cM{{\mathcal M}}

\newcommand\scM{{\mathscr M}}

\newcommand\scO{{\mathscr O}}

\newcommand\mvector{\boldsymbol}

\newcommand\vb{\mvector{b}}
\newcommand\vc{\mvector{c}}
\newcommand\vd{\mvector{d}}

\newcommand\vf{\mvector{f}}

\newcommand\vp{\mvector{p}}
\newcommand\vq{\mvector{q}}

\newcommand\vv{\mvector{v}}

\newcommand\vx{\mvector{x}}

\newcommand\vz{\mvector{z}}
\newcommand\vA{\mvector{A}}
\newcommand\vB{\mvector{B}}
\newcommand\vC{\mvector{C}}

\newcommand\vE{\mvector{E}}
\newcommand\vF{\mvector{F}}

\newcommand\vR{\mvector{R}}

\newcommand\vX{\mvector{X}}

\newcommand\vZ{\mvector{Z}}
\newcommand\valpha{\mvector{\alpha}}

\newcommand\vgamma{\mvector{\gamma}}
\newcommand\vGamma{\mvector{\Gamma}}

\newcommand\vomega{\mvector{\omega}}

\newcommand\vzero{\mvector{0}}
\newcommand\vveps{\mvector{\varepsilon}}
\newcommand\field{\mathbb}

\newcommand\C{\field{C}}
\newcommand\Z{\field{Z}}
\newcommand\M{\field{M}}
\newcommand\N{\field{N}}
\newcommand\Q{\field{Q}}

\newcommand\vspan{\operatorname{span}}

\newcommand\diag{\operatorname{diag}}

\newcommand\grad{\operatorname{grad}}
\newcommand\Sol{\operatorname{Sol}}

\newcommand\trdeg{\operatorname{tr.\!deg}}

\newcommand\id{\operatorname{\mathrm{Id}}}

\newcommand\gal{\operatorname{Gal}}
\newcommand\rmd{\mathrm{d}}

\newcommand\VE[1]{\ensuremath{\mathrm{VE}_{#1}}}
\newcommand\EX[1]{\ensuremath{\mathrm{EX}_{#1}}}
\newcommand\PV[1]{\ensuremath{\mathrm{PV}{#1}}}
\newcommand\rmi{\mathrm{i}\mspace{1mu}}
\newcommand\rme{\mathrm{e}}

\newcommand\Dt{\frac{\mathrm{d}\phantom{t} }{\mathrm{d}\mspace{1mu}
t}}

\newcommand\pder[2]{\dfrac{\partial #1 }{\partial #2}}

\newcommand\abs[1]{\lvert #1 \rvert}

\makeatletter
\@ifpackageloaded{polski}{%
\let\tg\@undefined 
\let\ctg\@undefined 
\let\tgh\@undefined 
\let\ctgh\@undefined 
\let\nwd\@undefined 
\let\arc\@undefined
\let\arsin\@undefined  

\DeclareMathOperator{\tg}{tg}
\DeclareMathOperator{\ctg}{ctg}
\DeclareMathOperator{\tgh}{tgh}
\DeclareMathOperator{\ctgh}{cth}
\DeclareMathOperator{\nwd}{nwd}
\newcommand\arc[1]{\operatorname{arc#1}}
\newcommand\arsin{\arc{\sin}}
}{%
\DeclareMathOperator{\tg}{tg}
\DeclareMathOperator{\ctg}{ctg}
\DeclareMathOperator{\tgh}{tgh}
\DeclareMathOperator{\ctgh}{cth}
\DeclareMathOperator{\nwd}{nwd}
\newcommand\arc[1]{\operatorname{arc#1}}
\newcommand\arsin{\arc\sin}
}{}
\makeatother

\newcommand\bfi[1]{{\bfseries\itshape{#1}}}

\newcommand\mtext[1]{\quad\text{#1}\quad}
\newcommand\mnote[1]{\marginpar{\tiny{#1}}}
\newcommand\defset[2]{\left\{{#1}\;\vert \;\; {#2} \,\right\}}

\begin{document}
%
%-----------------------------------------
%
\mathtoolsset{%
mathic,centercolon
}
%
%---------------------------------------
%
\thispagestyle{empty}
\vspace*{1em}
\begin{center}
  \Large\textsc{Integrability of Hamiltonian
    systems with homogeneous potentials of degrees $\pm 2$. An application of
    higher order variational equations}
\end{center}
\vspace*{0.5em}
\begin{center}
\large  Guillaume Duval$^1$ and Andrzej J.~Maciejewski$^2$
\end{center}
\vspace{1.0em}
\hspace*{2em}\begin{minipage}{0.8\textwidth}
\small
$^1$Laboratoire de Mathématiques et d'Informatique (LMI),\\
INSA de Rouen,
Avenue de l'Université, \\
76 801 Saint Etienne du Rouvray Cedex,
France. \\
E-mail: guillaume.duval@insa-rouen.fr\\[1em]
$^2$J. Kepler Institute of Astronomy,
 University of Zielona G\'ora, \\
 Licealna 9, PL-65--417,  Zielona G\'ora, Poland.\\
E-mail: maciejka@astro.ia.uz.zgora.pl
\end{minipage}\\[2em]
\vspace*{1.5em}
%\centerline{\color{Red} Version of: \today}
%\vspace{1.5em}
\today

{   \small \centerline{\textsc{Abstract}}}
\begin{minipage}{0.9\textwidth}
  \small The present work is the first of a serie of two papers, in
  which we analyse the higher variational equations associated to
  natural Hamiltonian systems, in their attempt to give Galois
  obstruction to their integrability.  We show that the higher
  variational equations $VE_{p}$ for $p\geq 2$, although complicated
  they are, have very particular algebraic structure. Preceisely they
  are solvable if $VE_{1}$ is virtually Abelian since they are
  solvable inductively by what we call the \emph{second level
    integrals}. We then give necessary and sufficient conditions in
  terms of these second level integrals for $VE_{p}$ to be virtually
  Abelian (see Theorem~\ref{thm:dudu}).  Then, we apply the above to
  potentials of degree $k=\pm 2$ by considering their $VE_{p}$ along
  Darboux points.  And this because their $VE_{1}$ does not give any
  obstruction to the integrablity. In Theorem~\ref{thm:oscillator}, we
  show that under non-resonance conditions, the only degree two
  integrable potential is the \emph{harmonic oscillator}. In contrast
  for degree $-2$ potentials, all the $VE_{p}$ along Darboux points
  are virtually Abelian (see Theorem~\ref{thm:k=-2}).
\end{minipage}

{\flushleft \textbf{key words}: Hamiltonian systems; integrability;
  differential Galois theory;}
\vspace{1.5em}

\section{Introduction}
\label{sec:intro}
Our aim in this paper is to developed methods and tools which allow
effective investigation of the integrability of complex Hamiltonian
through the analysis of the differential Galois group of their higher
order variational equations along a particular solution. This approach
is described in, e.g., \cite{Churchill:96::b,Audin:01::c}. The most
general result was obtained in ~\cite{Morales:07::}. It gives
necessary conditions for the complete meromorphic integrability of a
meromorphic Hamiltonian system defined on a complex analytic
symplectic manifold $M^{2n}$. These conditions are expressed in the
following result.
\begin{theorem}[Morales-Ramis-Sim\'o]
  \label{thm:MRS}
  Assume that a meromorphic Hamiltonian system is integrable in the
  Liouville sense with first integrals which are meromorphic in a
  connected neighbourhood $U$ of the phase curve $\Gamma$
  corresponding to a non-equilibrium solution, and are functionally
  independent in $U\setminus\Gamma$.  Then, for each $p\in\N$, the
  identity component $(G_p)^{\circ}$ of the differential Galois group
  $G_p$ of $p$-th order variational equations $\mathrm{VE}_p$ along
  $\vGamma$ is Abelian.
\end{theorem}
For background material, detailed exposition and proof of the above
theorem we refer the reader to~\cite{Morales:07::}, and \cite{Guy09}.

Numerous successful applications of this theorem were obtained just by
dealing with the first variational equations $\VE{1}$. For an overview
of these results see, e.g.,
\citep{Morales:10::a,Maciejewski:09::b}. The \VE{p} with $p\geq2$ are
much more complicated systems than \VE{1}. This is why, at present
time no systematic studies of the higher variational equations have
been made. The aim of these two papers is, among other things, to
extract some general structure of the \VE{p} for a certain wide class
of systems.

In order to apply the above theorem one has to know a particular
solution of the considered system.  Generally it is not easy to find
such a particular solution of a given system of nonlinear differential
equations. This is the reason why so many efforts have been devoted to
natural Hamiltonian systems with homogeneous potentials, for which we
can find particular solutions in a systematic way.

Indeed, let us consider a class of Hamiltonian systems with $n$
degrees of freedom generated by a natural Hamiltonian function of the
form
\begin{equation*}
  H = \frac{1}{2}\sum_{i=1}^np_i^2 +V(\vq), \qquad \vq=(q_1,\ldots,q_n),
\end{equation*}
where $V$ is a homogeneous function of degree
$k\in\Z^{\star}:=\Z\setminus\{0\}$.  For such systems, the
corresponding Hamilton's equations have the canonical form
\begin{equation}
  \label{eq:heqs}
  \Dt \vq = \vp, \qquad \Dt \vp = -V'(\vq),
\end{equation}
where $V'(\vq):=\grad V(\vq)$.

Generally equations~\eqref{eq:heqs} admit particular solutions of the
following form. Let a non-zero vector $\vd\in\C^n$ satisfy
\begin{equation*}
  V'(\vd)=\gamma\vd, \mtext{where}\gamma\in\C^{\star}.
\end{equation*}
Such a vector is called a proper Darboux point of the potential $V$.
It defines a two dimensional plane in the phase space $\C^{2n}$, given
by
\begin{equation*}
  \Pi(\vd):= \defset{(\vq,\vp)\in\C^{2n}}{ \vq =\varphi\vd, \  \vp=\psi\vd,
    \quad (\varphi,\psi)\in\C^2 }.
\end{equation*}
This plane is invariant with respect to system~\eqref{eq:heqs}.
Equations~\eqref{eq:heqs} restricted to $\Pi(\vd)$ have the form of
one degree of freedom Hamilton's equations
\begin{equation}
  \label{eq:one}
  \Dt \varphi = \psi, \qquad \Dt \psi = -\gamma\varphi^{k-1},
\end{equation}
with the following phase curves
\begin{equation*}
  \Gamma_{k,e} :=
  \defset{(\varphi,\psi)\in\C^2}{ \frac{1}{2}\psi^2 +\frac{\gamma}{k}\varphi^k=
    e}\subset\C^2, \qquad e \in\C.
\end{equation*}
In this way, a solution $(\varphi,\psi)=(\varphi(t),\psi(t))$
of~\eqref{eq:one} gives rise to solution
$(\vq(t),\vp(t)):=(\varphi\vd,\psi\vd)$ of equations~\eqref{eq:heqs}
with the corresponding phase curve
\begin{equation}
  \vGamma_{k,e} :=
  \defset{(\vq,\vp)\in\C^{2n}}{ (\vq,\vp)=(\varphi \vd,\psi\vd),\ (\varphi,\psi)\in  
    \Gamma_{k,e } }\subset\Pi(\vd).
\end{equation}
In this context significant obstructions to the integrability were
obtained just by dealing with the first variational equations
$\VE{1}$. However, as we proved in \citep{Maciejewski:09::}, for
$k=\pm2$, no obstruction can be found on the level of $\VE{1}$, since
the differential Galois group of \VE{1} is virtually Abelian.

For $k=2$ our main result in is the following theorem.
\begin{theorem}
  \label{thm:oscillator}
  Let $V$ be a homogeneous potential of degree $k=2$ satisfying the
  following assumptions:
  \begin{enumerate}
  \item it has a proper Darboux point, i.e., there exits a non-zero
    vector $\vd$, such that $V'(\vd)=\gamma \vd$ with
    $\gamma\in\C^{\star}$;
  \item the Hessian matrix $\gamma^{-1}V''(\vd)$ is diagonalisable
    with eigenvalues
    \begin{equation*}
      \lambda_1=\omega_1^2, \ldots,  \lambda_n=\omega_n^2,
    \end{equation*}
    such that $\omega_1,\ldots, \omega_n$ are $\Z$-linearly
    independent;
  \item $V$ is integrable in the Liouville sense.
  \end{enumerate}
  Then
  \begin{equation*}
    \label{eq:Vo}
    V(\vq)=\frac{1}{2}\vq^TV''(\vd)\vq.
  \end{equation*}
  In other words, in the eigenbasis of $V''(\vd)$, the Hamiltonian has
  the following form
  \begin{equation*}
    \label{eq:Hosc}
    H=\frac{1}{2}\sum_{i=1}^n\left(p_i^2+\omega_i^2q_i^2\right).
  \end{equation*}
\end{theorem}

In other words, this result shows that, under the non resonant
conditions given by point 2, a homogeneous potential of degree two is
integrable if and only if it is an harmonic oscillator.

In contrast, for applications it is important to have a result which
gives necessary conditions for the integrability without non-resonance
assumptions.  We formulate such results in Section~\ref{ssec:rescase}

As we already mentioned, for case $k=-2$, the second order variational
equations do not give any obstacles to the integrability.  Truly
amazing is the fact that the variational equations of any arbitrary
order do not give any obstruction to the integrability. More
precisely, we show the following.
\begin{theorem}
  \label{thm:k=-2}
  Let $V$ be a homogeneous potential of degree $k=-2$ which has a
  proper Darboux point $\vd$. Then, for each $p\in\N$, the
  differential Galois group of the higher variational equations \VE{p}
  along the phase curve associated with $\vd$ is Abelian.
\end{theorem}

In comparison with the previous result, this theorem for $k=-2$ is
quite anecdotic.  Nevertheless, it shows the local nature of the
Galois obstruction along a particular solution.  Indeed, the fact that
we do not get obstruction along a Darboux point, in the presnt case is
certainly intimately related with the integrability of homogeneous
potentials of degree $k=-2$ with two degrees of freedom.

The paper is organised as follows.  In Section~\ref{sec:gen} we
analyse the general structure of the \VE{p}, for $p\geq 2$.  This is
presented in the conpresent a potential of arbitrary degree
$k\in\Z$. The general \VE{p} are complicated linear system with second
member. The main goal of this Section is to show how some significant
sub-systems of \VE{p} can be extracted in order to apply the
Morales-Ramis-Sim\'o Theorem. This will be done explicitly for \VE{2},
since the study of these equations for general degrees $k$ will be the
main goal of the second of this paper.

From Section~\ref{sec:gen} it will be clear that the solutions of
\VE{p+1} are obtained by adding to the solutions of \VE{p} a certain
number of primitive integrals.  Formally it means the following. Let
$F_i/K$ be the Picard-Vessiot extension of \VE{i}. Then we have the
following tower of inclusions

\begin{equation*}
  K\subset F_1\subset F_2\subset \cdots \subset F_p\subset F_{p+1},
\end{equation*}
and $F_{p+1}/F_p$ is generated by certain number of elements $\Phi$,
such that $\Phi'\in F_p$. In Section~\ref{sec:theo} we show that if
the differential Galois group $F_{p+1}/K$ is virtually Abelian, then
we have a strong restrictions on the form of integrals $\Phi$, see
Theorem~\ref{thm:dudu}.

Section~\ref{sec:k=2} contains a proof of
Theorem~\ref{thm:oscillator}, and Section 5 contains a proof proof
Theorem~\ref{thm:k=-2}.

\section{General structure of  higher order variational equation}
\label{sec:gen}
\subsection{Solvability of higher order variational equations and
  second level integrals}
\label{ssec:hve}
Let us consider a system of differential equations
\begin{equation}
  \label{eq:ds}
  \Dt \vx=\vv(\vx), \qquad \vx\in U\subset \C^{m}, \qquad t\in\C,
\end{equation}
where $U$ is an open set, and the right hand sides
$\vv(\vx)=(v_1(\vx),\ldots, v_m(\vx))$ are holomorphic.  Let
$\vx_0(t)$ be a particular solution of this system. In a neighborhood
of $\vx_0(t)$ we represent $\vx$ in the following form
\begin{equation*}
  \vx= \vx_0(t) + \varepsilon \vx_1 +\frac{1}{2!}\varepsilon^2\vx_{2}+
  \frac{1}{3!}\varepsilon^3\vx_{3}+\cdots,
\end{equation*}
where $\varepsilon$ is a formal small parameter.  Inserting the above
expansion into both sides of equation~\eqref{eq:ds}, and equaling
terms of the same order with respect to $\varepsilon$ we obtain a chain 
of equations of the form
\begin{equation}
  \label{eq:hveds}
  \Dt \vx_p = \vA\vx_p +\vf_p(\vx_1,\ldots, \vx_{p-1}), \qquad p=1,2,
  \ldots, 
\end{equation}
where
\begin{equation}
  \label{eq:A}
  \vA=\pder{\vv}{\vx}\left( \vx_0(t) \right),
\end{equation}
and $\vf_1=\vzero$. In this settings, the $p$-th equation in the chain
is a linear non-homogeneous equation. Its non-homogeneous term depends on
general solution of first $p-1$ equations in the chain.  We called it
the variational equation of order $p$, and denote it by \VE{p}.

In this paper we work with Hamiltonian systems, so in particular
$m=2n$ and, system~\eqref{eq:ds} is the canonical Hamilton equations. Moreover, the matrix $\vA$ in~\eqref{eq:A} is an element of
$\mathfrak{sp}(2n, K)$, where $K$ is the differential ground field.

We assume that all the differential fields which appear in this paper are
of characteristic zero and their fields of constants is  $C=\C$.

Let $K$ be a differential field.  We denote $a'$, the derivative of
$a\in K$. Recall that the Picard-Vessiot ring of a Picard-Vessiot
extension $F/K$ is the set of elements of $F$ which are ``holonomic'' over
$K$. They form the set of elements of $F$ which are solutions of some non-trivial
linear differential equation with coefficients in $K$. We denote the
\emph{Picard-Vessiot ring} of a Picard-Vessiot extension $F/K$ by $T(F/K)$.
   
In our considerations an important role is played by the notion of ``levels of
integrals''.
\begin{definition}
  Let $K \subset F_1 \subset F_2$ be a ``tower of Picard-Vessiot
  extensions of $K$''.  By this we mean that $F_1/K$ and $F_2 /K$ are Picard-Vessiot
extensions with $F_1 \subset F_2$. An element $\Phi \in F_2$ is called an integral
  of second level with respect to $K$ iff $\Phi' \in T (F_1 / K)$. If
  moreover $\Phi' \in K$, i.e, if $\Phi$ is a primitive integral over
  $K$, we say that $\Phi$ is an integral of the first level.
\end{definition}
Obviously, a first level integral is an integral of second
level. Moreover, let us observe that any second level integral is
holonomic over $K$, hence it belongs to $T (F_2 / K)$.

Let us consider the following system of differential equations
\begin{subequations}
  \label{eq:ve1x}
  \begin{align}
    \label{eq:ve1x1}
    \vx_{1}' =&\vA \vx_1,\\
    \label{eq:ve1x2}
    \vx_2' =&\vA \vx_2+ \vB,
  \end{align}
\end{subequations}
where $\vA\in\mathfrak{sp}(2n,K)\subset\mathfrak{sl}(2n,K)$. Let us assume 
that the elements of the one column matrix $\vB$ belongs to the Picard-Vessiot ring $T (F_1 / K)$ of
equation~\eqref{eq:ve1x1}. According to equations ~\eqref{eq:hveds} and ~\eqref{eq:A}, the second order
variational equations of a general Hamiltonian system have such a form.  Let
$F_2\supset K$ be the Picard-Vessiot extension of the whole
system~\eqref{eq:ve1x}. We have that $K \subset F_1 \subset F_2$, and
the extension $F_2\supset F_1$ is generated by a certain number of
second level integrals.  To see this, let us take a fundamental matrix
$\vX_1\in\mathrm{Sp}(2n,F_1)$ of equation~\eqref{eq:ve1x1}. In order
to solve equation~\eqref{eq:ve1x2} we apply the classical variation of
constants method. That is, we look for a particular solution $\vX_2$ of the form
$\vX_2=\vX_1\vC$, where the column vector $\vC$ satisfies
\begin{equation}
  \label{eq:c}
  \vC' = \vX_1^{-1}\vB.
\end{equation}
Notice that the right hand sides of the above equation belong to $T
(F_1 / K)^{2n}$ since the fundamental matrix $X_1$ is unimodular.
Hence, we have
\begin{equation}
  \label{eq:X2}
  \vX_2 = \vX_1 \int \vX_1^{-1}\vB \Longleftrightarrow \vC= \int \vX_1^{-1}\vB.
\end{equation}
Thus, as claimed, the field $F_2$ is generated over $F_1$ by a certain
number of elements $\Phi$ such that $\Phi'\in T (F_1 /K)$. Precisely, $F_2/F_1$ 
is generated by the $2n$ entries of $\vC$.

Taking into account the above facts, and thanks to
equation~\eqref{eq:hveds}, we have proved the following.
\begin{lemma}
\label{lem:vepsol}
  Let us assume that $\VE{1}$ has  virtually Abelian differential
  Galois group $\gal(\VE{1})$, and let us consider
  the following tower of \PV{} extensions
  \begin{equation*}
    K\subset  \PV(\VE{1}) \subset  \PV(\VE{2}) 
    \subset \cdots \subset \PV(\VE{p}) \subset  \PV(\VE{p+1})
    \subset\cdots. 
  \end{equation*}
  Then the following statements hold true  for arbitrary $p\in\N$.
  \begin{enumerate}
  \item In each tower $K \subset \PV( \VE{p}) \subset \PV(\VE{p +
      1})$, the extension $ \PV( \VE{p}) \subset \PV(\VE{p + 1})$ is
    generated by the second level integrals.
    
  \item Each $\PV( \VE{p})$ is a solvable Picard-Vessiot extension.
  \end{enumerate}
\end{lemma}
The above lemma shows the particular structure of the $\VE{p}$.
 Although they are big complicated systems
 but, nevertheless, they are solvable. As a consequence, our main goal is going to find tractable conditions
 that distinguish the virtually Abelian ones between all these solvable systems.  
\begin{remark}
  \label{rem:nhsplit}
  In our further considerations we will use the following superposition principle.
  Let us consider a linear non-homogeneous system
  \begin{equation}
    \label{eq:lnh}
    \dot \vx = \vA\vx +\vB_1+\cdots +\vB_s,  \qquad \vx\in K^m,
  \end{equation}
  where $K$ is a differential field, and $\vA\in \M(m,K)$, $\vB_i\in
  K^m$ for $i=1,\ldots,s$. Let $\vx_1, \ldots, \vx_s\in K^m$, satisfy
  \begin{equation}
    \label{eq:lnhi}
    \dot \vx_i = \vA\vx_i +\vB_i  \mtext{for}i=1,\ldots, i.
  \end{equation}
  Then
  \begin{equation*}
    \widehat\vx= \vx_1+\cdots+\vx_s,
  \end{equation*}
  is a particular solution of~\eqref{eq:lnh}.

  Let $\widehat L/K$ and $L_i/K$ denote the Picard-Vessiot
  extensions of~\eqref{eq:lnh}, and~\eqref{eq:lnhi}, respectively.  By
  $\widehat G$ and $G_i$ we denote the corresponding differential
  Galois groups. By the above observation we have the following
  inclusion
  \begin{equation}
    \label{eq:in}
    \widehat L\subset L_1\cdots L_s,
  \end{equation}
  where the product denotes the composition of fields. Hence,
  $\widehat G$ is an algebraic subgroup of $G_1\times \cdots\times
  G_s$. Thus if $G_1,\ldots,G_s$ are virtually Abelian, then $\widehat
  G$ is virtually Abelian. Moreover, if in~\eqref{eq:in} we have the
  equality, then we have also the inverse implication.
\end{remark}
\subsection{Higher order variational equations along a Darboux point}

In this section we show the general structure of the second order
variational equations for the  Hamiltonian system~\eqref{eq:heqs} and
for a particular solution
$(\vq(t),\vp(t)):=(\varphi\vd,\dot\varphi\vd)$ associated to a proper
Darboux point $\vd$.   

We can  rewrite  equations~\eqref{eq:heqs} into Newton form
\begin{equation}
\label{eq:n}
\ddot \vq = \vF(\vq) 
\end{equation}
where $\vF(\vq)=-V'(\vq)$.    We put 
\begin{equation*}
  \vq= \vq_0 + \varepsilon\vq_1  +  \frac{1}{2!}\varepsilon^2\vq_2 +
 \frac{1}{3!}\varepsilon^3\vq_3 +\cdots
\end{equation*}
where $\vq_0=\varphi(t)\vd$ is 
the chosen  particular solution, end $\varepsilon$ is a formal small parameter.
Inserting the above expansion into equation~\eqref{eq:n} and comparing
terms of the same order with respect to $\varepsilon$ we obtain an
infinite sequence of   equation. The first of them
$\ddot \vq_0= \vF(\vq_0)$,  is identically  satisfied by
assumptions. For further purposes we  need the next three equations
which are the following
\begin{align}
\label{eq:ve1}
\ddot \vq_1= &\vF'(\vq_0)\vq_1, \\
\label{eq:ve2}
\ddot \vq_2= &\vF'(\vq_0)\vq_2+
 \vF''(\vq_0)(\vq_1, \vq_1),\\
\label{eq:ve3}
\ddot \vq_3= &\vF'(\vq_0)\vq_3+3
\vF''(\vq_0)(\vq_1,\vq_2)+\vF^{(3)}(\vq_0)(\vq_1,\vq_1,\vq_1),
\end{align}
From this we see that:
\VE{1} is a linear homogeneous equation given by~\eqref{eq:ve1}.  In
contrast, the \VE{p} for $p\geq2$ are  non-homogeneous linear
systems.But their linear part is the same as the one of \VE{1}. Moreover,
the second term in \VE{2} is a quadratic form in the solutions of \VE{1}.  
A bigger complexity appear in  \VE{3} given by~\eqref{eq:ve3}. The
second term in the right hand side of this equation is a bilinear form
in solution of \VE{1} and \VE{2}, while the third term is a cubic form
in the solutions of \VE{1}.

Notice that in the considered case $\vF$  is homogeneous of degree
$(k-1)$. This is why we have 
\begin{equation*}
\label{eq:hf}
\vF(\vq_0)=\vF(\varphi(t)\vd)=\varphi(t)^{k-1}\vF(\vd)=-\varphi(t)^{k-1}\vd,
\end{equation*}
and 
\begin{equation*}
\label{eq:hi}
\vF^{(i)}(\vq_0)=\varphi(t)^{k-1-i}\vF^{(i)}(\vd)
\end{equation*}

It seems that a  global investigation of the whole
system~\eqref{eq:ve1}-\eqref{eq:ve3} is too difficult. However, we are going to simplify 
its study by considering  subsystems of them. This will be done at the level of  $\VE{2}$ and $\VE{3}$. 
\begin{remark}
In the above calculations we implicitly assumed that the Darboux point
$\vd$ satisfies $V'(\vd)=\vd$. If we have a Darboux point $\vc$
satisfying $V'(\vc)=\gamma\vc$, then $\vd=\alpha \vc$ satisfies
$V'(\vd)=\alpha^{k-2}\gamma \vd$. Hence, if $k\neq 2$ we can choose
$\alpha$ in such a way that $\alpha^{k-2}\gamma=1$, and we do not loose
the generality. For $k=2$ we have to rescale the potential. If $V$  has a
Darboux point $\vc$ satisfying   $V'(\vc)=\gamma\vc$, then potential
$\widetilde V:=\gamma^{-1}V$ has the same integrability properties as
$V$, and $\widetilde V'(\vc)=\vc$.
\end{remark}
\subsection{Reduction procedure for $VE_{2}$ and $VE_{3}$}
\label{ssec:redu}
In order to simplify notations we fix the following conventions. To a
differential equation over a ground field $K$ we attach a certain
name, e.g., $\VE{2}$.  The ground field over which we consider this
equation is always clearly known from the context. Then, the
corresponding Picard-Vessiot extension and its differential Galois group
will be  denoted by  $\PV(\VE{2})$,  and by $\gal(\VE{2})$, respectively.

The main goal of this section is to prove that the differential Galois
group $\gal(\VE{2})$ is virtually Abelian if and only
if the differential Galois groups of a certain number of systems
extracted from $\VE{2}$ are virtually Abelian. Applications of these reductions are given in
Section ~\ref{ssec:rescase} below, they will be of crucial importance in the second part of the paper. 
But since the proofs of Theorems ~\ref{thm:oscillator} and~\ref{thm:k=-2} are independent of these considerations.
As a consequence, these section is not of crucial importance for their understanding.

Let us assume that the Hessian matrix $V''(\vd)$ is diagonalisable. Then,
without loss of  generality we can assume that it is diagonal, and
we put 
$V''(\vd)=\diag(\lambda_1,\ldots, \lambda_n )$.  Let us  denote also   
\begin{equation}
\label{eq:qk}
\vq_j=(q_{1,j},\ldots, q_{n,j}), \mtext{for} j\in \N.
\end{equation}
Then, the system of equations~\eqref{eq:ve1},  \eqref{eq:ve2}  and
~\eqref{eq:ve3} reads 
\begin{align}
\label{eq:dp1}
\ddot q_{i,1}=&-\lambda_i\varphi(t)^{k-2} q_{i,1}, \\
\label{eq:dp2}
\ddot q_{i,2}=& -\lambda_i\varphi(t)^{k-2} q_{i,2} +
\varphi(t)^{k-3}\Theta^{i}(\vq_1,\vq_1),\\
\label{eq:dp3}
\ddot q_{i,3}=&-\lambda_i\varphi(t)^{k-2} q_{i,3} +
3\varphi(t)^{k-3} \Theta^i(\vq_1,\vq_2) + \varphi(t)^{k-4}\,\Xi^i(\vq_1)
\end{align}
where $1\leq i \leq n$, and $\Theta^i$, and $\Xi^i$ are 
polynomials of their arguments 
\mnote{Notation!}
\begin{equation}
\label{eq:thi}
\Theta^i(\vq_1,\vq_2)= \sum_{\alpha, \beta=1}^{n}
\theta^i_{\alpha, \beta}\,q_{\alpha,1} q_{\beta,2},  \mtext{where }
\theta^i_{\alpha, \beta}=D_{\alpha, \beta}F_i(\vd), 
\end{equation}
and 
\begin{equation}
\label{eq:xi}
\Xi^i(\vq_1):= \sum_{\alpha,\beta,\gamma=1}^n 
\xi^i_{\alpha,\beta,\gamma}q_{\alpha,1}q_{\beta,1}q_{\gamma,1}, \mtext{where}
\xi^i_{\alpha,\beta,\gamma}=D_{\alpha,\beta,\gamma}F_{i}(\vd) . 
\end{equation}

The first order variational equations $\VE{1}$ is given by~\eqref{eq:dp1}.
I has the form of a direct product of independent equations. Thus we have a perfect splitting
of the problem at this level.  In order to perform effectively an analysis of
$\VE{2}$ and $\VE{3}$ we have to  split the problem into smaller subsystems.
We can do this in the following way. We set to zero all the variables
$q_{i, 1}$ except variable $q_{\alpha,1}$ in the 
system~\eqref{eq:dp1}--\eqref{eq:dp2}. We  get a system of $n$
independent subsystems  of $\VE{2}$ which we denote
$\VE{2,\alpha}^\gamma$, for $1\leq\gamma\leq n$ .
Such a system has the following form  
\begin{equation}
\label{eq:ve2ai}
\left.
\begin{aligned}
\ddot q_{\alpha,1}=&-\lambda_\alpha\varphi(t)^{k-2} q_{\alpha,1}, \\
\ddot q_{\gamma,2}=& -\lambda_\gamma\varphi(t)^{k-2} q_{\gamma,2} +
\varphi(t)^{k-3}\theta^\gamma_{\alpha,\alpha} q_{\alpha,1}^2,
\end{aligned} \quad \right\}
%\tag{$\VE{2,\alpha}^\gamma$}
\end{equation}
In a similar way, for two fixed indices $\alpha\neq\beta$,  we  distinguish $n$ other subsystems
$\VE{2,(\alpha,\beta)}^\gamma$ of $\VE{2}$.  They are 
subsystems of~\eqref{eq:dp1}--\eqref{eq:dp2} obtained by putting
$q_{i,1}=0$ except for $i\in \{\alpha,\beta\}$
They are of the following form  
\begin{equation}
\label{eq:ve2abi}
\left.
\begin{aligned}
\ddot q_{\alpha,1}=&-\lambda_\alpha\varphi(t)^{k-2} q_{\alpha,1}, \\
\ddot q_{\beta,1}=&-\lambda_\beta\varphi(t)^{k-2} q_{\beta,1}, \\
\ddot q_{\gamma,2}=& -\lambda_\gamma\varphi(t)^{k-2} q_{\gamma,2} +
\varphi(t)^{k-3}\left[\theta^\gamma_{\alpha,\alpha} q_{\alpha,1}^2 + 
2 \theta^\gamma_{\alpha,\beta} q_{\alpha,1}q_{\beta,1}+
\theta^\gamma_{\beta,\beta} q_{\beta,1}^2 \right]
\end{aligned} \quad \right\}
%\tag{$\VE{2,(\alpha,\beta)}^\gamma$}
\end{equation}

We fix three indices $\alpha, \beta, \gamma \in \{1,\ldots, n\}$ such
that $\alpha\neq \beta$.  From $\VE{2,(\alpha,\beta)}^\gamma$ we
extract a system $\EX{2,(\alpha,\beta)}^\gamma$ of the following form
\begin{equation}
  \label{eq:ve2abii}
  \left.
    \begin{aligned}
      \ddot q_{\alpha,1}=&-\lambda_\alpha\varphi(t)^{k-2} q_{\alpha,1}, \\
      \ddot q_{\beta,1}=&-\lambda_\beta\varphi(t)^{k-2} q_{\beta,1}, \\
      \ddot q_{\gamma,2}=& -\lambda_\gamma\varphi(t)^{k-2}
      q_{\gamma,2} + 2\varphi(t)^{k-3} \theta^\gamma_{\alpha,\beta}
      q_{\alpha,1}q_{\beta,1}.
    \end{aligned} \quad \right\}
  % \tag{$\EX{2,(\alpha,\beta)}^\gamma$}
\end{equation}                  %
Note that this system \bfi{is not} a subsystem of
$\VE{2,(\alpha,\beta)}^{\gamma}$.  According to
Remark~\ref{rem:nhsplit}, we have
\begin{equation}
  \label{eq:pvsabg}
  \PV(\VE{2,(\alpha,\beta)}^{\gamma})=\PV(\VE{2,\alpha}^{\gamma})\PV(\VE{2,\beta}^{\gamma})\PV(\EX{2,(\alpha,\beta)}^{\gamma}).
\end{equation}
The inclusion $\subset$ is evident. And the reverse inclusion follows
from
\begin{equation*}
  \PV(\EX{2,(\alpha,\beta)}^{\gamma})\subset
  \PV(\VE{2,\alpha}^{\gamma})\PV(\VE{2,\beta}^{\gamma})\PV(\VE{2,(\alpha,\beta)}^{\gamma})
\end{equation*}
which corresponds to the subtractions of particular solutions.  Hence
again, by Remark~\ref{rem:nhsplit}, we have that
$\gal(\VE{2,(\alpha,\beta)}^{\gamma})$ is virtually Abelian iff the groups
$\gal(\VE{2,\alpha}^{\gamma})$, $\gal(\VE{2,\beta}^{\gamma})$, and
$\gal(\EX{2,(\alpha,\beta)}^{\gamma})$ are virtually Abelian.

By the above facts and again by  Remark~\ref{rem:nhsplit}, we have proved
the following.
\begin{proposition}
  \label{pro:split}
  The differential Galois group $\gal(\VE{2})$ is virtually
  Abelian iff the groups $\gal(\VE{2,\alpha}^{\gamma})$ and
  $\gal(\EX{2,(\alpha,\beta)}^{\gamma})$ with $\alpha, \beta, \gamma
  \in \{1,\ldots, n\}$ and $\alpha\neq \beta$, are virtually Abelian.
\end{proposition}

\section{Second level integrals and virtually Abelian Galois groups}
\label{sec:theo}
According to Lemma~\ref{lem:vepsol}, we know that if $\gal(\VE{1})$ is virtually
Abelian,  then $\gal(\VE{p})$ is solvable for an arbitrary $p\in
\N$. Therefore, our main goal in this section is to find a necessary and
sufficient condition  which guarantee that  $\gal(\VE{p})$ is
virtually Abelian for $p\in\N$. 

From what follows, all the PV extensions $F / K$ will have the same
algebraically closed field of constants $C=\mathbb{C}$.

We have to analyse the following structure
\[ K \subset F_1 \subset F, \]
where $F / K$ and $F_1 / K$ are PV extensions, $\gal (F_1 / K)$
is virtually Abelian,  and $F /
F_1$ is generated by second level integrals $\Phi_1, \ldots,
\Phi_q$. We can assume that these integrals are independent over $F_1$.  Hence,
$H := \gal (F / F_1)$ is a vector group isomorphic to
$C^q$.   Therefore, we get the following exact sequence of algebraic groups
\begin{equation*}
  \begin{array}{llllllll}
    0 \longrightarrow & H =C^q & \longrightarrow & \gal (F / K) &
    \longrightarrow & \gal (F_1 / K) & \longrightarrow & 0 \label{eqexactseq}
  \end{array} .
\end{equation*}
As a consequence, the algebraic closure $\widetilde{K}$ of $K$ in $F_1$ coincides with the algebraic closure of $K$ in $F$.

In order to decide whether or not, $F / K$ is virtually Abelian, we shall use
the following result.
\begin{theorem}
  \label{thm:dudu}
  Let $K \subset F_1 \subset F$ be a tower of Picard-Vessiot
  extensions such that $F / F_1$ is generated by integral of second
  level over $K$. Then $G:=\gal(F/K)$ is virtually Abelian iff
  $G_1:=\gal(F_1/K)$ is virtually Abelian and any second level
  integrals $\Phi \in F$ can be expanded into the form
  \begin{equation*}
    \Phi = R_1 + J, \mtext{with} R_1 \in T (F_1 / K) \mtext{and}J' 
\in \widetilde{K},
  \end{equation*}
  where $\widetilde{K}$ is the algebraic closure of $K$ in $F_1$.
Moreover, a necessary condition for the virtual Abelianity of $G$ is that we get:
$\sigma(\Phi)-\Phi\in T(F_1/K)$ for all $\sigma\in G^{\circ}$.
\end{theorem}
We can  interpret this result is the following way. The fact that 
$\gal^{\circ} (F / K)$ is Abelian implies that any given second level integral can
be computed thanks to first level integral and exponential of integrals over
$\widetilde{K}$. 

\subsection{Proof of Theorem~\ref{thm:dudu}}

For sake of clarity, here we recall some classical facts about
Abelian algebraic groups and Abelian PV extensions. Most of these
results are contained in \cite{MR0376633}.

In a linear algebraic group $G$ any element $x \in G$ has a
Jordan-Chevalley decomposition $x = x_\mathrm{s} x_\mathrm{u} =
x_\mathrm{u} x_\mathrm{s}$ with $ x_\mathrm{u}$ and $ x_\mathrm{s}$
belonging to $G$. We shall denote by $G_\mathrm{s}$ and by
$G_\mathrm{u}$ the semi-simple and the unipotent parts of $G$.

\begin{proposition}
  \label{abgroups}
  \begin{enumerate}
  \item If $G$ is connected and Abelian then $G_\mathrm{s}$ and
    $G_\mathrm{u}$ are connected algebraic subgroups of $G$ and $G
    =G_\mathrm{s}\times G_\mathrm{u}$.
    
  \item In the previous context, $G_\mathrm{s}$ is a torus, i.e., is
    isomorphic to some $(G_{\mathrm{m}})^p$.
    
  \item A unipotent Abelian algebraic group is a vector group, i.e., is
    isomorphic to some $(G_{\mathrm{a}})^q$. Moreover, any algebraic
    group morphism between two of them is linear.
    
  \item If $0 \rightarrow (G_{\mathrm{a}})^q \rightarrow G \rightarrow
    G_1 \rightarrow 0$ is an exact sequence of connected Abelian
    algebraic group, then this sequence  splits. That is, $G$ contains
    a copy of $G_1$, and $G \simeq G_1 \times (G_{\mathrm{a}})^q$.
    
  \item A connected and Abelian linear algebraic group is isomorphic
    to some $ (G_{\mathrm{m}})^p\times (G_{\mathrm{a}})^q$.
  \end{enumerate}
\end{proposition}

\begin{proof}
  Point 1 is the theorem from~\cite{Humphreys:75::} page
  100. Point 2 follows from the theorem on page 104 in 
  \cite{Humphreys:75::}, and point 5 follows from 1, 2, and
  3.
  
  We prove point 3.  Let $G$ be unipotent and algebraic. It can be viewed as a closed
  subgroup of some $\mathrm{U} (n, C)$ the group of the $n \times n$
  unipotent upper triangular matrices. Hence, $\mathfrak{g}=
  \mathrm{Lie} (G)$ is a subalgebra of $\mathfrak{n}(n, C) =
  \mathrm{Lie} (\mathrm{U} (n, C))$, the Lie algebra of upper
  triangular nilpotent matrices. The exponential mapping $\exp :
  \mathfrak{n}(n, C) \rightarrow \mathrm{U} (n, C)$ is one to one with
  inverse the classical $\mathrm{Log}$ mapping. Therefore, $\exp$ is
  also a one-to-one mapping from $\mathfrak{g}$ to $G$. Since $\mathfrak{g}$ and
  $G$ are Abelian, $\exp$ is a group morphism, hence an
  isomorphism. More precisely, let $( N_1, \ldots, N_d)$ be a
  $C$-basis of $\mathfrak{g}$. The mapping $f : C^d
  \simeq(G_{\mathrm{a}})^d \rightarrow G$ given by
  \[ f (t_1, \ldots, t_d) := \exp (t_1 N_1 + \cdots + t_d N_d) = \exp
  (t_1 N_1) \times \cdots \times \exp (t_d N_d), \] is an isomorphism
  of algebraic groups.
  
  Now let $\varphi : G \simeq (G_{\mathrm{a}})^d \rightarrow G' \simeq
  (G_{\mathrm{a}})^q$ be a morphism of algebraic groups. From a
  $C$-basis $( e_1, \ldots, e_d)$ of $(G_{\mathrm{a}})^d$, we get
  $\forall \mathbf{t}= ( t_1, \ldots, t_d) \in \Z^d$,
  \begin{equation}
    \varphi (t_1 e_1 + \cdots + t_d e_d) = t_1 \varphi (e_1) + \cdots + t_d
    \varphi (e_d) . \label{eqlin}
  \end{equation}
  But since both terms of this formula are polynomial in $\mathbf{t}$,
  (\ref{eqlin}) holds for all $\mathbf{t} \in C^d$ since $\Z^d$ is
  Zarisky dense in $C^d$. Hence $\varphi$ is linear.
  
  It remains to prove point 4.  Let $f$ be the algebraic group
  morphism corresponding to the arrow $G \rightarrow G_1$. According
  to \citep[Th. 15.3, p. 99]{Humphreys:75::}, for all $x \in G$, $f
  (x)_{\mathrm{s}} = f (x_{\mathrm{s}})$ and $f (x)_{\mathrm{u}} = f
  (x_{\mathrm{u}})$. Hence, \ the restriction of $f$ to
  $(G)_{\mathrm{s}}$ must be surjective and possessing a trivial
  kernel since the semi-simple part of \ $(G_a)^q$ is trivial. As a
  consequence, $(G)_{\mathrm{s}} \simeq
  (G_1)_{\mathrm{s}}$. Similarly, we get an exact sequence for the
  unipotent parts $0 \rightarrow (G_{\mathrm{a}})^q \rightarrow
  (G)_{\mathrm{u}} \rightarrow (G_1)_{\mathrm{u}} \rightarrow
  0$. Thanks to point 1, we are reduced to prove that this latter sequence
  splits. But this is obviously true since by point 3 this sequence
  reduces to a sequence of linear spaces whose arrow are linear
  mappings.
\end{proof}

Now we are ready to prove Theorem~\ref{thm:dudu}.
\begin{proof}[Proof of Theorem~\ref{thm:dudu}]
  Let us assume that $G := \gal (F / K)$ is virtually Abelian. Let us 
  denote by $G_1 := \gal (F_1 / K)$. The proof will be the consequence
  of the two following steps

  {\textbf{First Step}}. The proof reduces to the case where $G$ is
  connected.  In this case $\widetilde{K} = K$, and $G_1$ is also
  connected and Abelian. According to the exact sequence
  (\ref{eqexactseq}) and Proposition \ref{abgroups}.(4),
  \[ G \simeq G_1 \times H \simeq G_1 \times (G_{\mathrm{a}})^q \text{
    for some } q \in \N. \] If $q = 0$ then $F = F_1$, and the Theorem
  follows. Now let us assume that $q \geq 1$. Let us set
  \[ M := F^{G_1} . \] Since $G_1 \vartriangleleft G$, the extension $M /
  K$ is Picard-Vessiot with Galois group
  $\gal (M / K) \simeq H \simeq (G_{\mathrm{a}})^q$. From \cite[Example
  1.141, p.32]{MR1960772}, there exist $J_1, \ldots, J_q \in M$ with
  $J_i' \in K$ such that $M = K (J_1, \ldots, J_q)$. The composition
  field $F_1 M$ is a differential sub-field of $F$ which has trivial
  stabilisator in $G$. Indeed, if $\sigma \in G$ fixes point-wise all
  the elements of $F_1 M$, then $\sigma \in H$ since it fixes all the
  elements belonging to $F_1$, and it also belongs to $G_1$ since it fixes
  all the elements belong to $M$. Therefore, $\sigma = \id$ and
  \[ F = F_1 M = F_1 (J_1, \ldots, J_q) . \] Since $\mathrm{tr}. \deg
  (F / F_1) = \dim_C (H) = q$, we deduce that $J_1, \ldots, J_q$ are
  algebraically independent over $F_1$.

  {\textbf{Second Step}} Let $\Phi$ be a second level integral. Since
  $\Phi' \in F_1$ and $\{J_1, \ldots, J_q \}$ is a transcendental
  basis of $F / F_1$, $\{\Phi, J_1, \ldots, J_q \}$ are $q + 1$
  algebraically dependant first level integrals over $F_1$. Therefore,
  according to the Ostrowski-Kolchin Theorem, see \cite{Kolchin:68::},
  there exist $(c_1, \ldots, c_q) \in C^q$ such that
  \[ \Phi - \sum_{i = 1}^q c_i J_i = R_1 \in F_1 . \] But $\Phi \in T
  (F / K)$ as well as $J := \sum_{i = 1}^q c_i J_i$.  Therefore, $R_1
  \in T (F_1 / K)$ and the first implication of the theorem is proved.

  {\textbf{Conversely}} If $G_1^{\circ}$ is Abelian and each second
  level integral has the form $\Phi = R_1 + J$, since $F / F_1$ is
  generated by those $\Phi$; it is also generated by the corresponding $J$ which are
  integral of first level w.r.t to $\tilde{K}$. But according to
   Exercise~1.41 on page 32 in \cite{MR1960772}, $F_1 / \tilde{K}$ is generated by
  some integrals of first level w.r.t to $\tilde{K}$ and exponentials of
  integrals. Therefore, the same happens for $F / \tilde{K}$ and
  $G^{\circ}$ is Abelian.

  {\textbf{Finally}}, if $\Phi$ can be written $\Phi=R_1+J$. Since $J'\in \tilde{K}$, 
  the conjugates of $J$ are of the form $\sigma(J)=J+h(\sigma)$ ,
  for some group morphism $h:G^{\circ}\rightarrow C$, for all $\sigma\in G^{\circ}$.
  As a consequence, \[\sigma(\Phi)-\Phi=\sigma(R_1)-R_1+h(\sigma)\in T(F_1/K).\]
This proves the necessary condition.

\end{proof}

\subsection{Additive and multiplicative  versions  of Theorem~\ref{thm:dudu}}

Here, we apply this result in two special cases
summarized in the following two lemmas.
\begin{lemma}
  \label{lem:add}
  Let $L / K$ and $L_1 / K$ be Picard-Vessiot extensions with $L_1 = K
  (I_1, \ldots, I_s)$, where $I_i' \in K$, for $i=1,\ldots,s$. Assume
  that there exists a $C$-linear combination $I$ of these integrals
  $I_i$ which is transcendental over K. Assume further that there
  exist $\Phi \in L$, and $w \in K^{\star}$ such that
  \[ \Phi' = wI, \mtext{ that is} \Phi = \int wI. \] Then, the Galois
  group $\gal(L / K)$ is virtually Abelian, implies that there exists a constant
  $c$ such that
  \[ c I -\int w \in K. \] Equivalently, $\Phi$ can be computed thanks
  to 
  a closed formula of the form
  \[ \Phi = P (I) + J \mtext{ with } P (I) := \frac{c}{2} I^2 + gI \in
  K [I], \mtext{ and } J' \in K. \]
\end{lemma}
\begin{proof}
  Since a $C$-linear combination of primitive integrals over $K$
  still is a primitive integral over $K$, without loss of generality,
  we may prove the result in the restricted case where:
  \[ L_1 = K (I), \mtext{with} I' \in K, I \not\in K. \] For all $
  \sigma \in G:=\gal(L/K)$, we have an additive formula of the form
  \[ \sigma (I) = I + c (\sigma), \mtext{ with } c (\sigma) \in C. \]
  Hence,
  \begin{equation*}
    \sigma (\Phi') = \sigma (wI) = w \sigma (I)= 
    w (I + c (\sigma))= \Phi' + c (\sigma) w.
  \end{equation*}
  So, there exists some constant $d (\sigma)$ such that
  \begin{equation}
    \sigma (\Phi) - \Phi = c (\sigma) \int w + d (\sigma) \label{eq1+}
  \end{equation}
  Now let us assume that $G$ is virtually Abelian. According to
  Theorem~\ref{thm:dudu}, for all $\sigma\in G^{\circ}$,
  \begin{equation*}
   \sigma(\Phi)-\Phi =c(\sigma) \int w +d(\sigma) \in K [I] = T (L_1 / K).
  \end{equation*}
  Hence, if we choose $\sigma_0 \in G^{\circ}$ such that $c(\sigma_0)=1$, we get that
  \begin{equation*}
    \int w \in K [I].
  \end{equation*}
 Therefore, the two
  primitive integrals $\int w$ and $I$ are algebraically dependant
  over $K$.  In such a case, by the Ostrowski-Kolchin theorem, there
  are two constants $(\alpha, \beta) \in C^2 \backslash \{0\}$ such
  that
  \[ \alpha \int w + \beta I \in K. \] But in this relation $\alpha$
  cannot be zero since $I$ is transcendental, and the result follows.
  The converse implication follows by integration by part.
\end{proof}
\begin{lemma}
  \label{lem:mul}
  Let $L / K$ and $L_1 / K$ be Picard-Vessiot extensions with $L_1 = K
  (E_1, \ldots, E_s)$, where $E_i' / E_i \in K$, for $i=1,\ldots, s$.
  Assume that $L$ contain one element $\Phi$ of the following form
  \begin{equation}
    \label{eq:intm}
    \Phi:=\int \sum_{i=1}^rw_iM_i
  \end{equation}
  where $w_i\in K^{\star}$, and each
  \begin{equation*}
    M_i=M_i(E_1,\ldots, E_s)\in\C[E_1,E_1^{-1},\ldots,E_s,E_s^{-1} ], 
  \end{equation*}
  is a monomial, for all $1\leq i\leq r$. Moreover, $M_1, \ldots, M_r$ are
  not mutually proportional, i.e.,
  \begin{equation*}
    \frac{M_i}{M_j}\not\in K \mtext{for} i\neq j.
  \end{equation*}
  Then, we have:
  \begin{enumerate}
  \item Each separated integral $\Phi_i:=\int w_i M_i\in L$.
  \item If the extension $L/L_1$ is generated by $\Phi $, then $L/K$
    is virtually Abelian implies that for each $1\leq i \leq r$, there exists
    $c_i\in\C$, such that
    \begin{equation*}
      \frac{\Phi_i + c_i}{M_i} \in K.
    \end{equation*}
  \end{enumerate}
\end{lemma}
\begin{proof}
  First we prove point $1$.  Let us take $M=E_1^{n_1}\cdots E_s^{n_s}
  $ with $n_1, \ldots, n_s \in \Z$. If for $\sigma\in\gal(L/K)$ we
  have
  \begin{equation*}
    \sigma(E_i)=\rho_i(\sigma)E_i, \mtext{for} 1\leq i \leq s,
  \end{equation*}
  then
  \begin{equation*}
    \sigma(M)=\chi_M(\sigma)M,
  \end{equation*}
  where $\chi_M$ is a character of $\gal(L/K)$ given by
  \begin{equation*}
    \chi_M(\sigma):=\rho_1(\sigma)^{n_1}\cdots \rho_s(\sigma)^{n_s}.
  \end{equation*}
  We denote $\chi_i=\chi_{M_i}$ for $1\leq i\leq r$.  As by assumption
  $M_i/M_j\not\in K$ for $i\neq j$, we have
  \begin{equation*}
    \chi_i \neq \chi_j, \mtext{for} i\neq j, 
  \end{equation*}
  Now, for each $\sigma\in\gal(L/K)$ we have
  \begin{equation*}
    \sigma(\Phi')=\sum_{j=1}^r\sigma(w_jM_j)=\sum_{j=1}^r
    \chi_j(\sigma)w_jM_j= \sum_{j=1}^r
    \chi_j(\sigma) \Phi'_j.
  \end{equation*}
  Hence, we get that for each $\sigma\in\gal(L/K)$
  \begin{equation}
    \label{eq:sF}
    \sigma(\Phi)-\sum_{j=1}^r
    \chi_j(\sigma) \Phi_j \in \mathbb{C} \Longrightarrow\sum_{j=1}^r
    \chi_j(\sigma) \Phi_j \in L.
  \end{equation}
  Now, since $\chi_i\neq\chi_j$ for $i\neq j$, by the Artin-Dedekind
  lemma, the characters $\chi_1, \ldots, \chi_r$ are $\C$-linearly
  independent. As a consequence there exist $r$ elements $\sigma_1,
  \ldots, \sigma_r$ of $\gal(L/K)$ such that $r\times r$ matrix
  $[\chi_j(\sigma_i)]$ is invertible. If we write $r$ corresponding
  equations~\eqref{eq:sF} and invert this system we obtain that each
  $\Phi_j\in L$, for $1\leq j\leq r$.

  Now, we prove point 2. As in the proof of the previous Lemma, we can
  assume without loss of the generality that here, $\Phi=\int w E$,
  where 
  \[ L_1 = K (E), \mtext{with} E' / E \in K, \mtext{and} \trdeg (L_1 /
  K) = 1. \] For each $\sigma \in G:=\gal(K/L)$, we have a
  multiplicative formula of the form
  \[ \sigma (E) = \lambda (\sigma) E, \mtext{with} \lambda (\sigma)
  \in \C^{\star} . \] Hence,
  \begin{equation*}
    \sigma (\Phi')  =  \sigma (wE) = w \sigma (E)= \lambda (\sigma) \Phi' .
  \end{equation*}
  So, there exists some constant $d (\sigma)$ such that
  \begin{equation}
    \sigma (\Phi) = \lambda (\sigma) \Phi + d (\sigma) \label{eq1m}
  \end{equation}
  From~\eqref{eq1m}, the linear representation of $G^{\circ}$ in $V :=
  \vspan_{\C} \{1, \Phi\}$, gives a morphism of algebraic groups
  \[
  \rho : G^{\circ} \rightarrow G_{\mathrm{aff}}, \qquad \sigma \mapsto
  \rho (\sigma) :=
  \begin{bmatrix}
    1 & d (\sigma)\\
    0 & \lambda (\sigma)
  \end{bmatrix}.
  \]
  The affine group $G_{\mathrm{aff}} \simeq G_{\mathrm{m}} \ltimes
  G_{\mathrm{a}}$ is a solvable, non Abelian group of dimension
  two. Since $\lambda (\sigma)$ cover all non zero constant values,
  the image $\rho(G^{\circ})$ contains a maximal torus which is
  therefore isomorphic to $G_\mathrm{m}$.

  So, if we assume that $G^{\circ}$ is Abelian, then the image
  $\rho(G^{\circ})$ must be a maximal torus. As a consequence, all
  the matrices $\rho (\sigma)$ with $\sigma \in G^{\circ}$ can be
  simultaneously diagonalisable in a fixed basis
  \[ \left\{ [1,0]^T, [c,1]^{T}\right\}, \mtext{for some} c \in C. \]
  Hence, by direct computation we get that the $\rho (\sigma)$ are of
  the form
  \[ \rho (\sigma) = \begin{bmatrix}
    1 & d (\sigma)\\
    0 & \lambda (\sigma)
  \end{bmatrix}
  \mtext{with} d (\sigma) = c [\lambda (\sigma) - 1] .
  \]
  Taking this into account, we see that for all $\sigma \in
  G^{\circ}$, formula \eqref{eq1m} reads
  \begin{equation*}
    \sigma (\Phi)  =  \lambda (\sigma) \Phi + c [\lambda (\sigma) -
    1]. 
  \end{equation*}
  So
  \begin{equation*}
    \sigma (\Phi + c)  = \lambda (\sigma) [\Phi + c], 
  \end{equation*}
  an this gives
  \begin{equation*}
    \sigma \left( \frac{\Phi + c}{E}\right)  =  \frac{\Phi + c}{E} .
  \end{equation*}

  This means that $(\Phi + c) / E$ is algebraic over $K$. But now
  $\Phi + c$ is a first level integral with respect to $K (E)$. That
  it is algebraic over this field implies that it belongs to $K
  (E)$. But now, the restriction morphism $G^{\circ} \rightarrow \gal
  (K (E) / K) \simeq G_{\mathrm{m}}$ is surjective.  Therefore, $(\Phi
  + c) / E$ is fixed by $G^{\circ}$ implies that it belongs to $K$.
\end{proof}
Let us observe that, curiously, the last lemma was proved without any
reference to Theorem~\ref{thm:dudu}. 

\section{Potential of degree $k=2$ and their  $\mathrm{VE}_p$}  
\label{sec:k=2}
In this section we prove Theorem~\ref{thm:oscillator} and  also give 
some results concerning resonant cases. 

\subsection{The $\mathrm{VE}_p$ for $k=2$}
According to Proposition~\ref{pro:split}, to find the necessary and
sufficient conditions which guarantee that $\gal(\VE{2})$ is
virtually Abelian, we have to find such conditions for
$\VE{2,\alpha}^{\gamma}$ and $\EX{2,(\alpha,\beta)}^{\gamma}$, with
$\alpha, \beta, \gamma \in \{1,\ldots, n\}$ such that $\alpha\neq
\beta$.

Since we consider the case $k=2$, a particular solution $\varphi(t)$ with
energy $e=1/2$ satisfies ${\dot\varphi}^2+ \varphi^2=1$.  Thus, we can
take, e.g., \[\varphi(t)=\sin t.\] We denote by $\vGamma$ the corresponding phase curve in
$\C^{2n}$. The first order variational equation
\begin{equation*}
  \ddot \vq_1:= -  \varphi^{k-2}V''(\vd)\vq_1=  -  V''(\vd)\vq_1
\end{equation*}
is a matrix second order equation with constant coefficients over the differential ground field
\[K=\C(\varphi,\dot\varphi)=\C\left(\rme^{\rmi t}\right).\] Hence, we
can rewrite it in the following form 
\begin{equation}
  \label{eq:ve11k2}
  {\dot \vx}_1=\vA \vx_1, \mtext{where} 
  \vA:=\begin{bmatrix}
    \phantom{-}\vzero_n & \vE_n \\
    -V''(\vd) &\vzero_n
  \end{bmatrix}\in \mathrm{sp}(2n,\C).
\end{equation}
It is easy to show that the matrix $\vA$ has eigenvalues $\pm \rmi \omega
$, where $\omega^2=\lambda$, when $\lambda$ span the eigenvalues of
$V''(\vd)$.  Thus, the entries of the fundamental matrix $\vX_1$ of
equation~\eqref{eq:ve11k2} belong to a ring of the form
\begin{equation*}
  R_1:=\C\left(\rme^{\rmi t}\right)\left[\rme^{\pm\rmi\omega_1
      t};\ldots; 
\rme^{\pm\rmi\omega_n t};t \right],
\end{equation*}
where $\omega_1^2, \ldots, \omega_n^2$ are the eigenvalues of $V''(\vd)$.

With these notations and assumptions, $\VE{2,\alpha}^{\gamma}$ can be
rewritten into the following form 
\begin{equation}
\label{eq:ve2ag}
\left.
\begin{aligned}
\ddot x=&-\omega^2_\alpha x, \\
\ddot z=& -\omega_\gamma^2 z+
\frac{\theta^\gamma_{\alpha,\alpha}}{\sin t} x^2,
\end{aligned} \quad \right\}
%\tag{$\VE{2,\alpha}^\gamma$}
\end{equation}
where, to simplify notations, instead of blind variables $
q_{\alpha,1}$  and  $q_{\gamma,2}$,   we introduce  $x$ and $z$. In a
similar way we rewrite   $\EX{2,\alpha}^{\gamma}$ into the form 
\begin{equation}
\label{eq:vx2abg}
\left.
\begin{aligned}
\ddot x=&-\omega_\alpha x, \\
\ddot y=&-\omega_\beta y, \\
\ddot z=& -\omega_\gamma z +
2 \frac{\theta^\gamma_{\alpha,\beta}}{\sin t} xy.
\end{aligned} \quad \right\}
%\tag{$\EX{2,(\alpha,\beta)}^\gamma$}
\end{equation}

The last equations in~\eqref{eq:ve2ag} and~\eqref{eq:vx2abg} have the
same  form 
\begin{equation}
  \label{eq:ob}
  \ddot z = -\omega^2 z  + \frac{b(t)}{\sin t}, 
\end{equation}
with $\rme^{\rmi\omega t}$  and $b(t)$ belonging to  $R_1$.  

Now, we consider equation~\eqref{eq:ob} over $K=\C\left(\rme^{\rmi t}\right)$. Our aim
is to compute its Picard-Vessiot extension $L/K$.  Let $L_1$ be the
Picard-Vessiot extension of $K$ containing: $b(t)$, $\dot b(t)$, and
all the solutions of $\ddot z=-\omega^2 z$. Since $b(t)$ is holonomic over
$K$, it belongs to the Picard-Vessiot ring $T(L_1/K)$. Moreover, the
extension $L/L_1$ is generated by the second level integrals over $K$.
The form of these integrals is described into the following property. 
\begin{lemma}
  \label{lem:b}
With the above notations the following statements hold true.
  \begin{enumerate}
  \item If $\omega=0$, then $L/L_1$ is generated by 
    \begin{equation*}
      \int \frac{b(t)}{\sin t} \,\rmd t, \mtext{and} \int \frac{tb(t)}{\sin t} \,\rmd t. 
    \end{equation*}
  \item If $\omega\neq 0$, then $L/L_1$ is generated by 
    \begin{equation*}
      \int \frac{\rme^{\rmi\omega t}b(t)}{\sin t} \,\rmd t,
      \mtext{and} 
\int \frac{\rme^{-\rmi\omega t}b(t)}{\sin t} \,\rmd t.
    \end{equation*}
  \end{enumerate}
\end{lemma}
\begin{proof}
  We rewrite equation~\eqref{eq:ob} as the following non-homogeneous
  linear system
  \begin{equation}
    \label{eq:ys}
    \begin{bmatrix}
      \dot z_1\\
      \dot z_2
    \end{bmatrix} = \vA
    \begin{bmatrix}
      z_1\\
      z_2
    \end{bmatrix}
    + \vb(t)
  \end{equation}
  where
  \begin{equation*}
    \vA=
    \begin{bmatrix}
      0 & 1 \\
      -\omega^2 & 0
    \end{bmatrix}, \mtext{and} \vb= \frac{1}{\sin t}
    \begin{bmatrix}
      0 \\
      b(t)
    \end{bmatrix}
  \end{equation*}
Let us denote by $\vZ$ a fundamental matrix of solutions of the
homogeneous part of~\eqref{eq:ys}.  The variation of constants gives a
particular solution of the form $\widehat \vz =\vZ\vc$ with
$\vc=[c_1,c_2]^T$ satisfying 
\begin{equation*}
\Dt \vc =\vZ^{-1}\vb.
\end{equation*}
Since $\vZ\in \mathrm{GL}(2,L_1)$, and $L/L_1$ is generated by the two
entries of $\widehat \vz$, it is also generated by $c_1$ and $c_2$. % $a$

 Now, if $\omega=0$, then 
  \begin{equation*}
    \vZ =  \begin{bmatrix}
      1& t\\
      0  & 1 
    \end{bmatrix}, \mtext{and}  \dot\vc =\vZ^{-1}\vb=  \frac{b(t)}{\sin t} 
    \begin{bmatrix} 
      - t \\
      \phantom{-}1
    \end{bmatrix}
  \end{equation*}
  If $\omega\neq 0$, then
  \begin{equation*}
    \vZ:=\begin{bmatrix}
      \phantom{i}\rme^{\rmi\omega t } &  \phantom{-i} \rme^{-\rmi\omega t }  \\
      \rmi \omega\rme^{\rmi\omega t } &  -\rmi \omega\rme^{-\rmi\omega t } 
    \end{bmatrix}, 
\mtext{and}\dot\vc =\vZ^{-1}\vb=  \frac{\rmi b(t)}{2\omega\sin t} 
    \begin{bmatrix} 
      -\rme^{-\rmi\omega t } \\
      \;\rme^{\rmi\omega t }
    \end{bmatrix}
  \end{equation*}
Now, in both cases, the claim follows easily.
\end{proof}

%
% ------------------------------------------------------
\subsection{Proof of Theorem~\ref{thm:oscillator}}

\subsubsection{Taylor expansion of $V$ around the Darboux points.}
We can  assumed that the considered Darboux point $\vd$ satisfies
$V'(\vd)=\vd$. Moreover, we also assume that
$V''(\vd)=\diag(\omega_1^2, \ldots,\omega_{n-1}^2 ,\omega_n^2)$.
From the Euler identity we easily deduce that
$V''(\vd)\vd=\vd$. Hence, $1$ is an eigenvalue of $V''(\vd)$ with the
corresponding eigenvector $\vd$. This implies that one of $\omega_i^2$
is one, and we can assume that $\omega_n=1$. Thus,
$\vd=[0,\ldots,0,1]^T$.
We also assumed that $V$ is  analytic around this Darboux point. So,
we have the following 
 Taylor expansion 
\begin{equation*}
V(\vq)=\sum_{\abs{\valpha}\geq0}^{\infty}\frac{1}{\valpha!} 
\partial^{\valpha}V(\vd)\cdot{\widetilde\vq}^{\valpha}
\end{equation*}
where $\widetilde\vq=\vq-\vd=(q_1, \ldots, q_{n-1},q_n-1)$. In the
above 
we use the standard multi-index notations. That is, 
\begin{equation*}
\valpha=(\alpha_1,\ldots, \alpha_n)\in\N^n, \qquad
\abs{\valpha}:=\sum_{i=1}^n\alpha_{i}, \qquad \valpha!=\alpha_1!
\cdots \alpha_n!, \qquad
\vq^{\valpha}:=q_1^{\alpha_1}\cdots q_n^{\alpha_n}, 
\end{equation*}
and 
\begin{equation*}
\partial^{\valpha}:=\partial_1^{\alpha_1}\cdots \partial_n^{\alpha_n}, 
\mtext{where} \partial_i^{\alpha_i}:=
\frac{\partial^{\alpha_i}\phantom{l}}{\partial
q_i^{\alpha_i}}.
\end{equation*}
Taking into account that  $V'(\vd)=\vd$,  we also  have  by the Euler
identity   $V(\vd)=1/2$,
and the above expansion can be written in the form 
\begin{equation}
\label{eq:ext}
V(\vq)= \frac{1}{2}\left(\omega_1^2q_1^2+\cdots+
  \omega_{n-1}^2q_{n-1}^2 +q_n^2\right) +\sum_{\abs{\valpha}\geq3}
\frac{1}{\valpha!} \partial^{\valpha}V(\vd)\cdot{\widetilde\vq}^{\valpha}.
\end{equation}
Thus, our aim is to prove that 
\begin{equation*}
 \partial^{\valpha}V(\vd) =0, \mtext{for all} \abs{\valpha}=m\geq 3,
\end{equation*}
 when the potential is integrable. This will be done by induction with respect to $m$. 
 
Let us return to the equations of motion written in Newton form given by equation
~\eqref{eq:n}, i.e., 
\begin{equation*}
\ddot\vq=\vF(\vq), 
\end{equation*}
where $ \vF(\vq) = -V'(\vq)$, i.e., 
%$^{}$
\begin{equation*}
F_i(\vq)=-\pder{V}{q_i}(\vq)=-\partial^{\vveps_i}V(\vq).
\end{equation*}
In the above, the $i$-th component of multi-index $\vveps_i$ is equal
one, and all remaining vanish.  Thus, we have 
\begin{equation}
\label{eq:alFi}
\partial^{\valpha}F_i(\vd) =- \partial^{\valpha+\vveps_i}V(\vd), 
\end{equation}
or equivalently
\begin{equation*}
\partial^{\valpha}V(\vd) = -\partial^{\valpha - \vveps_i}F_i(\vd).
\end{equation*}

\subsubsection{The induction procedure.}
Let $p\geq 2$. We assume that \VE{p} has  the following simple form 
\begin{equation}
\label{eq:spf}
\left. 
\begin{aligned}
\ddot x_i&=-\omega_i^2 x_i, \qquad 1\leq i \leq n, \\
\ddot y_j&= -\omega_j^2 y_j +
\sum_{\abs{\valpha}=p}\frac{\xi^j_{\valpha}\vx^{\valpha}}{\sin^{p-1}(t)},
\qquad
1\leq j\leq n,
\end{aligned} \quad \right\}
\end{equation}
where $\vx=(x_1,\ldots, x_n)$, and 
\begin{equation*}
\xi^j_{\valpha}:= \frac{1}{\valpha!}\partial^{\valpha}F_j(\vd).
\end{equation*}
According to~\eqref{eq:ve2ag} and~\eqref{eq:vx2abg}, we notice that $\VE{2}$ has  such form. For $p\geq 3$, according
to Section~\ref{ssec:redu}, \VE{p} will have such a simple form  iff 
\begin{equation*}
\partial^{\vgamma}F_j(\vd)=0, \mtext{with} 1\leq j\leq n 
\end{equation*}
whenever, 
\begin{equation*}
2\leq \abs{\vgamma}<p. 
\end{equation*}
Our aim is therefore to prove that  $\xi^j_{\valpha}=0$, for all  $1\leq j\leq n
$, and all $\valpha$ such that $\abs{\valpha}=p$.

\paragraph{First Step.}
Let $K=\C(\rme^{\rmi t})$ be our ground field,  and $L:=\PV(\VE{p})$ be
the Picard-Vessiot extension of \VE{p} over $K$. The field $L$
contains 
\begin{equation*}
L_1:=\PV(\VE{1})=K(\rme^{\rmi \omega_1 t}, \ldots,
 \rme^{\rmi \omega_{n-1} t}).
\end{equation*}
According to  Lemma~\ref{lem:b}, $L$ contains second level integrals of the
form 
\begin{equation*}
\Phi_j=\int\sum_{\abs{\valpha}=p}\frac{\xi^j_{\valpha}\vx^{\valpha}\rme^{\rmi
  \omega_jt}}{\sin^{p-1}(t)}, \mtext{for all} 1\leq  j \leq n.
\end{equation*}
In particular, if we choose 
\begin{equation*}
\vx=(x_1,\ldots, x_n)=(\rme^{\rmi\omega_1t},\ldots, \rme^{\rmi
\omega_nt}), \mtext{with} \omega_n=1,
\end{equation*}
then 
\begin{equation*}
\vx^{\valpha}=
\rme^{\rmi(\alpha_1\omega_1 + \cdots+\alpha_n\omega_n) t} =
\rme^{\rmi(\valpha\cdot\vomega) t},
\end{equation*}
and $L$ contains integrals of the form 
\begin{equation}
\label{eq:fijs}
\Phi_j=\sum_{\abs{\valpha}=p}\xi^j_{\valpha} \int 
\frac{1}{\sin^{p-1}(t)}\rme^{\rmi(\valpha\cdot\vomega+\omega_j)t}=
\sum_{\abs{\valpha}=p}\xi^j_{\valpha}T_{p-1}^{(\valpha\cdot\vomega+\omega_j)},
\end{equation}
where the ``trigonometric integrals'' $T_{n}^{(\omega)}$ are defined and studied in the Appendix.

In accordance with Lemma~\ref{lem:mul}, let us consider the following
monomials 
\begin{equation*}
M_{\valpha}:= \rme^{\rmi(\valpha\cdot\vomega+\omega_j)t}
\end{equation*}
with fixed $j$. We claim that they are not mutually proportional.  We
prove this claim by contradiction.  Thus, let $\valpha$  and
$\valpha'$ be two different multi-indices, and assume that
\begin{equation*}
\frac{M_{\valpha'}}{M_{\valpha}}= 
\rme^{\rmi(\valpha' -\valpha)\cdot\vomega t}\in K=\C(\rme^{\rmi t}).
\end{equation*}
This implies that 
\begin{equation*}
(\valpha' -\valpha)\cdot\vomega = (\alpha'_1-\alpha_1)\omega_1+\cdots 
(\alpha'_{n-1}-\alpha_{n-1})\omega_{n-1}+(\alpha'_n-\alpha_n)=m\in\Z.
\end{equation*}
But, since $\omega_1,\ldots, \omega_{n-1}$ are $\Z$-linearly
independent, this implies that 
\begin{equation*}
\alpha'_i=\alpha_i \mtext{for} 1\leq i <n. 
\end{equation*}
Since $\abs{\valpha}=\abs{\valpha'}$, we also have 
$\alpha'_n=\alpha_n$, and in effect $\valpha'=\valpha$. This is a
contradictory with the assumption that  $\valpha'\neq\valpha$. In this way
we proved our claim. 

Now, according to the first point of Lemma~\ref{lem:mul} we have that
each term in the sum~\eqref{eq:fijs} is an element of $L$, i.e., 
\begin{equation*}
\xi^j_{\valpha}T_{p-1}^{(\valpha\cdot\vomega+\omega_j)}\in L.
\end{equation*}
Moreover, according to the second point of  Lemma~\ref{lem:mul}, we
know that if $\gal(L/K)$ is virtually Abelian, then there exists a
constant $c$ such that 
\begin{equation*}
\frac{\xi^j_{\valpha}T_{p-1}^{(\valpha\cdot\vomega+\omega_j)}+c}{
\rme^{\rmi(\valpha\cdot\vomega+\omega_j)t}}\in K=\C(\rme^{\rmi t}).
\end{equation*}
It follows that $\xi^j_{\valpha}T_{p-1}^{(\valpha\cdot\vomega+\omega_j)}$
is meromorphic on $\C$. When $p=2$, by Lemma A1 $
T_{p-1}^{(\valpha\cdot\vomega+\omega_j)}$ is never meromorphic, and hence
$\xi^j_{\valpha}=0$. When $p\geq 3$, by Lemma A4, $T_{p-1}^{(\valpha\cdot\vomega+\omega_j)}$ 
is not meromorphic unless that  
\begin{equation*}
\valpha\cdot\vomega+\omega_j\in\Z.
\end{equation*}
Since $\valpha\in\N^n$,  the condition $\valpha\cdot\vomega+\omega_j\in\Z$  is 
satisfied only when $j=n$ and
$\valpha=(0,\ldots,0,p)$. 

Summarizing, we proved that if $\valpha\neq(0,\ldots,0,p)$, or $j\neq
n$, then  $\xi^j_{\valpha}=0$.
 
\paragraph{Second Step.}
It still has to be shown that $\xi^j_{\valpha}=0$, for $j=n$ and
$\valpha=\valpha_0=(0,\ldots,0,p)$.  We know that 
\begin{equation*}
\xi^j_{\valpha}=\frac{1}{\valpha!}\partial^{\valpha} F_j(\vd) = 
-\frac{1}{(\valpha+\vveps_j)!}\partial^{\valpha + \vveps_j} V(\vd),
\end{equation*}
and so, $\xi^n_{\valpha_0}$ is proportional to
$\partial^{p+1}_nV(\vd)$. Notice that  $\partial^{p}_nV(\vq)$ is
homogeneous of degree $(2-p)$. Thus, we have thanks to Euler's identity
\begin{equation*}
\sum_{i=1}^n q_i \partial_i\partial^{p}_nV(\vq) = (2-p)\partial^{p}_nV(\vq).
\end{equation*}
Evaluating both sides of the above identity at $\vq=\vd=(0,\ldots,0,1)$  we obtain 
\begin{equation}
\label{eq:dvp1}
\partial^{p+1}_nV(\vd) = (2-p)\partial^{p}_nV(\vd).
\end{equation}
By induction we assume that $\VE{{p-1}}$ reduces to $\VE{1}$, so
$\xi^n_{(0,\ldots,0, p-1)} \simeq \partial^{p}_nV(\vd) = 0 $. Hence we
also have   $\xi^n_{(0,\ldots,0, p)} =0$.

Summarizing, we prove that $\xi^j_{\valpha}=0$ for all $1\leq j \leq
n$, and all $\valpha$ such that $\abs{\valpha}=p$.

\paragraph{Third Step.} Finally, notice that we show that  
\begin{equation*}
\partial^{\valpha} V(\vd) =0, \mtext{with} 3\leq\abs{\valpha}\leq
p+1. 
\end{equation*}
Thus  $\VE{p+1}$ has the same form as the assumed form of $\VE{p}$. Hence the
induction hypothesis  go on. 

We can conclude the proof of the theorem  observing that the first 
step of the induction procedure applies  to $\VE{2}$, that is when
$p=2$. Indeed the first step goes without any change.  Moreover, in
the second step with $p=2$,  the identity~\eqref{eq:dvp1} gives 
\begin{equation*}
\partial^{3}_nV(\vd) = 0\, \partial^{2}_nV(\vd)=0.
\end{equation*}
So, we have the correct initial step of  induction. This ends the proof.

\subsection{The Group $\gal(\VE{2})$ in resonant cases when $k=2$}
\label{ssec:rescase}
The following two lemmas give the necessary and sufficient conditions
which guarantee that $\gal(\VE{2,\alpha}^{\gamma})$ and
$\gal(\EX{2,(\alpha,\beta)}^{\gamma})$ are virtually Abelian.
\begin{lemma}
\label{lem:va1} 
The group  $\gal(\VE{2,\alpha}^{\gamma})$ is virtually Abelian iff  either
$\theta^{\gamma}_{\alpha,\alpha}=0$,  or
$\theta^{\gamma}_{\alpha,\alpha}\neq0$, and 
$\omega_{\alpha},\omega_{\gamma}\in\Q^{\star}$.
\end{lemma}
\begin{lemma}
\label{lem:va2} 
The group $\gal(\EX{2,(\alpha,\beta)}^{\gamma})$  is virtually Abelian iff  either
$\theta^{\gamma}_{\alpha,\beta}=0$,  or
$\theta^{\gamma}_{\alpha,\beta}\neq0$, and 
$\omega_{\alpha},\omega_{\beta},\omega_{\gamma}\in\Q^{\star}$.
\end{lemma}
\subsubsection{Proof of Lemma~\ref{lem:va1}}
In this subsection we denote by $L/K$ the Picard-Vessiot extension of
$\VE{2,\alpha}^{\gamma}$.  The field $L$ contains the field $L_1$,
where $L_1/K$ is the Picard-Vessiot extension associated with the
system of homogeneous equations $\ddot x +\omega_{\alpha}^2x=0$ and
$\ddot z +\omega_{\gamma}^2z=0$.

If $\theta_{\alpha,\alpha}^{\gamma}=0$, then $\VE{2,\alpha}^{\gamma}$
is a subsystem of \VE{1}, therefore $\gal(L/K)$ is virtually Abelian.

Let us observe also that if both $\omega_{\alpha}$ and
$\omega_{\gamma}$ belongs to $\Q^{\star}$, then the field
$L_1=K(\rme^{\rmi \omega_{\alpha} t }, \rme^{\rmi \omega_{\gamma} t
})$ is a finite extension of $K$. Moreover, the field $L/L_1$ is
generated by a finite set of integrals of first level with respect to
$L_1$. As a consequence, we have
\begin{equation*}
  \gal^{\circ}(L/K)=\gal(L/L_1)\simeq G_{\mathrm{a}}^p,
\end{equation*}
for a certain positive integer $p$. Hence, $\gal(L/K)$ is virtually
Abelian.

The above observations show that we have to consider the case
$\theta_{\alpha,\alpha}^{\gamma}\neq0$. But in this case, without loss
of the generality, we can assume that
$\theta_{\alpha,\alpha}^{\gamma}=1$. To investigate whether or not 
$\gal{L/K}$ could be virtually Abelian, we need a description of the
field $L_1$, and of the extension $L/L_1$. These informations are collected in the following table.
\begin{equation}
  \label{tab:1}
  \begin{array}{ccccc}
    \toprule
    \text{case} &  (\omega_{\alpha} ,\omega_{\gamma}) & L_1 & \vspan_{\C}(x^2)
    &\sin(t)\dot \Phi \\
    \midrule 
    \addlinespace[0.5em]
    1&  (0, 0) & \{t\} & \{1,t,t^2\} &
    \{1,t,t^2,t^3 \}\\
    \addlinespace[0.5em]
    2&  (0, \star) & \left\{t, \rme^{\rmi\omega_{\gamma}t}\right\}&
    \{1,t,t^2\} &
    \left\{t^d\rme^{\pm\rmi\omega_{\gamma}t}, 0\leq d\leq 2\right\}\\
    \addlinespace[0.5em]
    3& (\star, 0) & \{t, \rme^{\rmi\omega_{\alpha}t}\} &
    \left\{1,\rme^{\pm2\rmi \omega_{\alpha}t}\right\} &
    (1,\rme^{\pm2\rmi \omega_{\alpha}t}, t,t\rme^{\pm
      2\rmi \omega_{\alpha}t}) \\
    \addlinespace[0.5em]
    4&  (\star,\star) & \left\{\rme^{\rmi\omega_{\alpha}t},\rme^{\rmi\omega_{\gamma}t}\right\} &
    \left\{1,\rme^{\pm2\rmi \omega_{\alpha}t}\right\}&
    \left\{\rme^{\pm\rmi \omega_{\gamma}t}    ,\rme^{\rmi(\pm  \omega_{\gamma}\pm2\omega_{\alpha})t}\right\}\\\addlinespace[0.3em]
    \bottomrule
  \end{array}
\end{equation}
Let us explain the origin of this table.  The exact form of the field
$L_1$ depends on whether $\omega_{\alpha}$ or $\omega_{\gamma}$
vanish. This gives us four non-equivalent cases. The symbol $\star$ in the
second column means that the corresponding $\omega$ is not zero.  The column marked by
$L_1$ contains generators of the extension $L_1/K$.

Now, in the formulae of Lemma~\ref{lem:b}:
\[ \frac{b (t)}{\sin (t)} = \frac{x^2}{\sin (t)}, \] where $x$ belongs to
the solution space $\vspan_{\C}(x) := \Sol\{ \ddot{x} +
\omega_{\alpha}^2 x = 0\}$. This is why, $b (t)$ spans the second
symmetric power of this vector space. That is:
\[ \vspan_{\C}(b) := \vspan_{\C}(x)^{2\,\circledS},
\]
is a three dimensional vector space over $\C$. Now, the  second
level integrals $\Phi$ given by Lemma~\ref{lem:b} are such that for
$\omega_{\gamma} = 0$,
\begin{equation*}
  \Phi = \int \frac{b (t)}{\sin (t)}\, \rmd t \mtext{or}
  \Phi = \int \frac{tb (t)}{\sin (t)} \,\rmd t 
\end{equation*}
Hence, $\sin (t) \dot{\Phi}$, spans the vector space
\[ \vspan_{\C}(b) + t \cdot \vspan_{\C}(b) . \] Similarly, from
Lemma~\ref{lem:b} again, if $\omega_{\gamma} \neq 0$, we get that
$\sin(t) \dot{\Phi}$ spans
\[ \rme^{\mathrm{i} \omega_{\gamma} t} \cdot \vspan_{\C}(b) + \rme^{-
  \mathrm{i} \omega_{\gamma} t} \cdot \vspan_{\C}(b) . \] These
considerations immediately give the last column of the table.

As a consequence, the second level integrals $\Phi$ involved in the
problem are of three types
\begin{equation}
  \label{eq:tt}
  T_{\omega} := \int \frac{\rme^{\mathrm{i} \omega t}}{\sin (t)} \,\rmd
  t, \quad P_d := \int \frac{t^d}{\sin (t)}\, \rmd t, \mtext{or} M_{d, \omega} := \int \frac{t^d
    \rme^{\mathrm{i} \omega t}}{\sin (t)}\, \rmd t,
\end{equation}
where $\omega\in\C$ and $n\in\N$.  These are the ``trigonometric'',
``polynomial'' and the ``mixed'' integrals. Observe further that $T_0
= P_0 = M_{0, 0}$.  According to Lemma A.1 of the Appendix, all these
integrals are not meromorphic on $\C$ and are therefore transcendental over $K=\C(\rme^{\rmi t})$.

Now, in the three first cases of table~\eqref{tab:1},
\[ I := t \in L_1 . \] We shall therefore apply
Lemma~\ref{lem:add}, to show that in these three cases,
$\gal(L / K)$ is not virtually Abelian.

\paragraph{Case 1.} According to the Table, we can write $L_1 = K
(I)$.  Moreover, we also have 
\[ P_1 = \int \frac{t}{\sin (t)} dt = \int \dot{P_0} I \in L,
\mtext{with} \dot{P_0} = \frac{1}{\sin (t)} \in K. \] So, $\Phi = P_1$
is a second level integral with respect to $K$ of the form appearing
in Lemma~\ref{lem:add}. Hence, according to it, we get the
implication
\[
\gal(L / K) \mtext{is virtually Abelian}\Longrightarrow \quad cI -
\int \dot{P_0} = c t - P_0 \in K,
\]
for a certain $c\in\C$.  In particular this implies that $P_0$ is
meromorphic on $\C $ since $K =\C (e^{\mathrm{i} t})$ is contained
into $\scM(\C )$ the field of meromorphic functions on $\C $. But
this is not true thanks to Lemma A.1, so $\gal(L / K)$ is not
virtually Abelian.

\paragraph{Case 2.} Here $\omega_{\alpha} = 0$ but $\omega_{\gamma}
\neq 0$.  Hence, we can write $L_1 = K (t, e^{\mathrm{i}
  \omega_{\gamma} t}) = \widehat{K} (I)$ with $\widehat{K} := K
(e^{\mathrm{i} \omega_{\gamma} t})$. We therefore get the implication
\[ \gal(L / K) \mtext{is virtually Abelian} \Longrightarrow\quad
\gal(L / \widehat{K}) \mtext{is virtually Abelian,} \] and we are
reduced to proving that $\gal(L / \widehat{K})$ is not virtually
Abelian.

By setting $I = t$ again, according to table~\eqref{tab:1},
\[ M_{1, \omega_{\gamma}} = \int \frac{t\rme^{\mathrm{i}
    \omega_{\gamma} t}}{\sin (t)} dt = \int \dot{T}_{\omega_{\gamma}}
I \in L, \mtext{with} {\dot T}_{\omega_{\gamma}} = \frac{^{}
  e^{\mathrm{i} \omega_{\gamma} t}}{\sin (t)} \in \widehat{K} , \] is
a second level integral with respect to $\widehat{K}$ of the form
appearing into Lemma~\ref{lem:add}. So we can apply this result
to the field extension $L / \widehat{K}$. We therefore get the
implication
\[ \gal(L / \widehat{K}) \mtext{is virtually Abelian } \Longrightarrow
\quad c I -\int \dot{T}_{\omega_{\gamma}} = c t - T_{\omega_{\gamma}}
\in \widehat{K} ,
\]
for a certain $c\in\C$.  But again this is not true, since
$\widehat{K} \subset \mathcal{M}(\C )$ and $T_{\omega_{\gamma}}$ is
not meromorphic on $\C$.

\paragraph{Case 3.}Is very similar to the previous one. Here we set
$\widehat{K} := K (e^{\mathrm{i} \omega_{\alpha} t})$, and now
\[ \Phi = M_{1, 2 \omega_{\alpha}} \in L. \] Therefore, if $\gal(L /
\widehat{K})$ is virtually Abelian, then $T_{2 \omega_{\alpha}}$ is
meromorphic, which again is not true.

\paragraph{Case 4.}Here we shall apply Lemma~\ref{lem:mul}. In this case,
$L_1 = K (\exp (\mathrm{i} \omega_{\alpha} t), \exp (\mathrm{i}
\omega_{\gamma} t))$ and
\[ L / L_1 = L_1 (T_{\pm 2 \omega_a \pm \omega_{\gamma}}, T_{\pm
  \omega_{\gamma}}), \] is generated by eight integrals of first level
with respect to  $K$. Our task is now to prove the implication
\[ \gal(L / K) \text{ virtually Abelian } \Rightarrow \omega_{\alpha},
\omega_{\gamma} \in \Q ^{\ast} . \] Let us therefore assume that there
exist $\omega_0 \in \{\omega_{\alpha}, \omega_{\gamma} \}$, with
$\omega_0 \not\in \Q $. The corresponding exponential $\exp
(\mathrm{i} \omega_0 t)$ is not algebraic over $K$, and the
quasi-toroidal extension $L_1 / K$ has the transcendental degree
$\geq 1$.  Moreover, there exists $\omega \in \{\pm 2
\omega_{\alpha} \pm \omega_{\gamma}, \pm \omega_{\gamma} \}$ which is
not rational. So we can apply  Lemma~\ref{lem:mul}, with $\Phi :=
T_{\omega}$ in order to get the implication
\[ \gal(L / K) \text{ virtually Abelian } \Rightarrow
\frac{T_{\omega}+c}{\exp (\mathrm{i} \omega t)} \in K. \] But this
further implies that $T_{\omega}$ is meromorphic on $\C$.
Since it is not true, the claim follows.

These considerations finishes the proof of  Lemma~\ref{lem:va1}. 

\subsubsection{Proof of Lemma~\ref{lem:va2} }

Here, as in the proof of Lemma~\ref{lem:va1}, we denote by $L / K$ the PV extension of
$\ensuremath{\operatorname{EX}}^{\gamma}_{2, \alpha, \beta}$. We denote also
by $L_1 / K$ the PV extension obtained by adding to $K$ the three solutions
spaces 
\[ 
\begin{split}
\vspan_{\C}(x) :=&
\operatorname{Sol}\{ \ddot{x} + \omega_{\alpha}^2 x = 0\},\\\
\vspan_{\C}(y) := &\operatorname{Sol}\{
\ddot{y} + \omega_{\beta}^2 y = 0\}, \\ \vspan_{\C}(z):=&\operatorname{Sol}\{
\ddot{z} + \omega_{\gamma}^2 z = 0\}. 
\end{split}
\]
As before, $L_1 \subset L$.

When $\omega_{\alpha}, \omega_{\beta}, \omega_{\gamma} \in
\Q ^{\ast}$, $L_1 = K (\exp (\mathrm{i} \omega_{\alpha} t), \exp
(\mathrm{i} \omega_{\beta} t), \exp (\mathrm{i} \omega_{\gamma} t))$ is a
finite extension of $K$, and $L / L_1$ is generated by a finite set of first
level integrals with respect  to $L_1$. Hence, $\gal(L /
K)$ is virtually Abelian since its connected component is a vector group.

When $\theta_{\alpha, \beta}^{\gamma} = 0$, then
$\ensuremath{\operatorname{EX}}^{\gamma}_{2, \alpha, \beta}$ is a subsystem of
$\ensuremath{\operatorname{VE}}_1$, hence is virtually Abelian. As a
consequence, it is enough  to prove that away from these two cases,
$\gal(L / K)$ is not virtually Abelian. We may
therefore assume without loss of generality that $\theta_{\alpha,
\beta}^{\gamma} = 1$. The collected information about $L_1 / K$ and $L / L_1$
are contained in the following table.
\begin{equation}
\label{tabc}
\begin{array}{ccccc}
     \toprule
     &( \omega_{\alpha} , \omega_{\beta} , \omega_{\gamma}) & L_1 / K &
     \vspan_{\C}(xy)
     & \sin (t) \dot{\Phi}\\
      \midrule 
    \addlinespace[0.5em]
     1 & (0 , 0 , 0) & \{t\} & \{1, t, t^2 \} & 1, t, t^2, t^3\\
      \midrule 
    \addlinespace[0.5em]
     2 &( 0 , 0 , \star) & \{t, \rme^{\mathrm{i} \omega_{\gamma} t}\} & \{1, t,
     t^2 \} & t^d \rme^{\pm \mathrm{i} \omega_{\gamma} t}, d= 0,1,2 
  \\
      \midrule 
    \addlinespace[0.5em]
     3 & (\star, 0 , 0) & \{\rme^{\mathrm{i} \omega_{\alpha} t}, t\} & \{
     \rme^{\pm \mathrm{i} \omega_{\alpha} t}, t \rme^{\pm \mathrm{i}
     \omega_{\alpha} t}\} & t^d \rme^{\pm \mathrm{i} \omega_{\alpha} t}, d=0,1,2
    \\
      \midrule 
    \addlinespace[0.5em]
     4 & (\star, 0 , \star) & \{\rme^{\mathrm{i} \omega_{\alpha} t}, t, 
     \rme^{\mathrm{i} \omega_{\gamma} t}\} & \{\rme^{\pm \mathrm{i} \omega_{\alpha}
     t}, t \rme^{\pm \mathrm{i} \omega_{\alpha} t}\} & t^d \rme^{\pm
     \mathrm{i} \omega_{\alpha} t \pm \mathrm{i} \omega_{\gamma} t},
   d=0, 1 \\
      \midrule 
    \addlinespace[0.5em]
     5 &( \star, \star , 0) & \{\rme^{\mathrm{i} \omega_{\alpha} t}, 
\rme{\mathrm{i} \omega_{\beta} t}, t\} & \{\rme^{\pm \mathrm{i}
     \omega_{\alpha} t \pm \mathrm{i} \omega_{\beta} t}\} & t^d \rme^{\pm
     \mathrm{i} \omega_{\alpha} t \pm \mathrm{i} \omega_{\beta} t}, d=0,1
     \\
      \midrule 
    \addlinespace[0.5em]
     6 & (\star , \star ,\star) & \{\rme^{\mathrm{i}
     \omega_{\alpha} t}, \rme^{\mathrm{i} \omega_{\beta} t}, \rme^{\mathrm{i}
     \omega_{\gamma} t}\} & \{\rme^{\pm \mathrm{i} \omega_{\alpha} t \pm
     \mathrm{i} \omega_{\beta} t}\} & \rme^{\pm \mathrm{i} \omega_{\alpha} t
     \pm \mathrm{i} \omega_{\beta} t \pm \mathrm{i} \omega_{\gamma} t}\\
     \bottomrule
   \end{array} 
\end{equation}

As in Table~\eqref{tab:1}, the cases are distinguished by the
vanishing of at least one component in the triple
$(\omega_{\alpha},\omega_{\beta},\omega_{\gamma})$. Thus, in general
we have eight possibilities. However, due to the form
$\ensuremath{\operatorname{EX}}^{\gamma}_{2, \alpha, \beta}$, and the
fact that $\theta_{\alpha, \beta}^{\gamma} = \theta_{\beta,
  \alpha}^{\gamma}$, the indicies $\alpha$ and $\beta$ play a symmetric
role.  This is why we  have to deal with only six distinct cases. The
Column $L_1 / K$ gives a set of generators of this extension. 
The column $\vspan_{\C}(xy)$ gives a basis of the vector space
\[ \vspan_{\C}(b) =\vspan_{\C}(xy)
   =\vspan_{\C}(x) \circledS
   \vspan_{\C}(y) . \]
Observe that in contrast with Table~\eqref{tab:1},  this $\C $ vector space can
be of dimension 3, or 4. Now the elements of the last column are computed
thanks to the same rule as in Table~\eqref{tab:1}.

From the last column, one can easily see that  all the second level
integrals which appear in this context are of the type described in the proof
of Lemma~\ref{lem:va1}. We can
therefore perform the proof in a similar way as before. Precisely, for Cases 1
to 5, $\gal(L / K)$ is not virtually
Abelian thanks to Lemma~\ref{lem:add}. In Case 6, the use of Lemma~\ref{lem:mul},
gives the implication
\[ \gal(L / K) \text{ virtually Abelian }
   \Rightarrow \omega_a, \omega_{\beta}, \omega_{\gamma} \in
   \Q ^{\ast} . \]
This ends the proof of this Lemma.

\section{Potentials of degree $k=-2$ and, Proof of Theorem~\ref{thm:k=-2}}

In our paper \cite{Maciejewski:09::}, we found a correspondence between the
first variational equations $\VE{1}$ of potentials of degree $
k$ and of degree $-k$.  Moreover, we used it to get the implication
\[ \mathrm{Gal} (\mathrm{VE}_1) \mtext{virtually Abelian for}
   k = 2 \Rightarrow \mathrm{Gal} (\mathrm{VE}_1) \mtext{virtually
     Abelian for} k = - 2. \]
The present theorem shows that the later correspondence cannot hold at the
level of the $\mathrm{VE}_p$ for $p \geqslant 2$. And  because of this
we get Galois
obstruction along the Darboux points for $k = 2$, although there is no such
obstruction for potentials of degree $k = - 2$.

\subsection{Ingredients for the proof}

Let $k = - 2$, and $\vd \in \C^n \setminus \{0\}$ be a proper Darboux
point satisfying $V' (\vd) = \vd$.  The corresponding phase curve is 
$\vGamma_{e} = \{(\varphi \vd, \dot{\varphi} \vd) \in \mathbb{C}^{2 n}
\}$, where 
\[ 
\dot{\varphi}^2 =
   e + \dfrac{1}{\varphi^2}, 
 \]
and  $e\in\C$ is a fixed value of the energy.  At first glance we analyse the solutions of the
equation
\begin{equation}
\label{eq:sub}
\ddot{x} = -\frac{\lambda}{\varphi^4} x, \qquad \lambda\in\C,
\end{equation}
which is a sub-equation of \VE{1}. 
The explicit form $\varphi$
depends on $e$. To investigate solutions of equation~\eqref{eq:sub} we
introduce a parameter $\omega$ defined by  $\lambda=1-4\omega^2$, and
two functions of time $I$, and $E_{\omega}$. They are defined in the
table below.
\begin{equation}
\label{tab:fei} 
\begin{array}{ccc}
     \toprule
      & e = 0 & e \neq 0\\
     \midrule
     \varphi & \varphi := \sqrt{2 t} & \varphi := \sqrt{e t^2 - \dfrac{1}{e}}\\
     \midrule
     E_{\omega}   & E_{\omega} := t^{\omega} &
     E_{\omega} : = \left( \dfrac{et - 1}{et + 1}\right)^{\omega}\\
     \midrule
     I & I := \log (t) & I := \log \left( \dfrac{et - 1}{et + 1}\right)\\
     \bottomrule
   \end{array} 
 \end{equation}
 \begin{proposition}
 \label{prop:k=-2} 
  Let us consider equation~\eqref{eq:sub} with  $\lambda=1-4\omega^2$.
  If $\omega = 0$, then its solution space  is generated by $\varphi$ and $\varphi
  I$; otherwise it is generated by $\varphi
  E_{\omega}$ and $\varphi E_{- \omega}$.
  
  The functions $I$ and $E_{\omega}$ satisfy the following relations
  \[ \dot{E}_{\omega} = \frac{2 \omega}{\varphi^2} E_{\omega}, \qquad \dot{I} =
     \frac{2}{\varphi^2} . \]
\end{proposition}
The proof of this proposition follows by direct computation and will be left to the reader. 

In \citep{Maciejewski:09::} we have shown that, for $k=-2$, the first
variational equations \VE{1} are solvable. Hence,
equation~\eqref{eq:sub} is solvable because it is a  sub-equation of
$\VE{1}$.  

Let us fix the  value of $e$, and denote by $K:=
\C (\varphi, \dot{\varphi})$. We consider the infinite dimensional
Picard-Vessiot extension  $F / K$  generated by the integral $I$ and
the exponentials 
$E_{\omega}$ with $\omega \in \C$. By Proposition~\ref{prop:k=-2},
$\gal{(F /K)}$ is virtually Abelian.
     
Let us consider the following vector sub-space of $\mathcal{B}$ of $F$:
\[ \mathcal{B} := \mathrm{Vect}_{\C} \{I^m E_{\omega} |m \in
   \N, \quad\omega \in \C\} . \]
We note here that this vector space contains $\C$ and remains stable by multiplication
since
\[( I^m E_{\omega} )\cdot (I^{m'} E_{\omega'}) = I^{m + m'} E_{\omega + \omega'}
   . \]

   The basic idea of our proof of Theorem~\ref{thm:k=-2} is to show
   that, with the use of vectorial notations, that we get the following inclusion
\[ \mathrm{Sol} (\VE{p}) \subset \varphi \mathcal{B}^n , \]
holds true for $p\in\N$.
Let us observe that such an inclusion will imply that $\mathrm{PV} (\mathrm{VE}_p)
\subset F$, and the result will follow. In order to get this inclusion we shall
need the two following lemmas.
\begin{lemma}
  \label{lem:int} For  $b (t) \in \mathcal{B}$, we have 
\[
 \Phi := \int \frac{b
  (t)}{\varphi^2} dt \in \mathcal{B}.
\]
\end{lemma}

\begin{proof}
  Clearly, it is enough to show this for the integrals 
  \[ 
\Phi^{(m)}_{\omega} :=\int \frac{I^m E_{\omega}}{\varphi^2} dt \mtext{for}
m\in\N, \mtext{and} \omega\in\C.
\]
  According to Proposition~\ref{prop:k=-2}, we have:
  \[ \Phi^{(m)}_{0} := \int \frac{I^m}{\varphi^2} dt = \int \frac{I^m 
     \dot{I}}{2} dt = \frac{I^{(m + 1)}}{2 (m + 1)} \in \mathcal{B}.
 \]
Moreover,  for $\omega \neq 0$, we also have
 \[
\Phi^{(0)}_{\omega}:= \int
  \frac{E_{\omega}}{\varphi^2} dt = \frac{E_{\omega}}{2 \omega} \in
  \mathcal{B}. 
\]
 Now, if $m \geqslant 1$, we obtain 
  \begin{eqnarray*}
    2 \omega \Phi^{(m)}_{\omega} & = & \int I^m  \frac{2 \omega
    E_{\omega}}{\varphi^2} = \int I^m  \dot{E}_{\omega}\\
    & = & I^m E_{\omega} - \int mI^{m - 1}  \dot{I} E_{\omega}\\
    & = & I^m E_{\omega} - \int mI^{m - 1}  \frac{2}{\varphi^2} E_{\omega}. 
  \end{eqnarray*}
As a result we get 
\begin{equation*}
 2 \omega \Phi^{(m)}_{\omega}  =  I^m E_{\omega} - 2 m \Phi^{(m - 1)}_{\omega},
\end{equation*}
 and the result follows by induction with respect to $m$.
\end{proof}
The following result extends Proposition~\ref{prop:k=-2}.
\begin{lemma}
\label{lem:k=-2} The solution space of the equation 
\begin{equation}
\label{eq:hv-2}
\ddot{x} = -\frac{ \lambda}{\varphi^4} x +
  \frac{b(t)}{\varphi^3}, \qquad \lambda\in\C, \quad b(t)\in\mathcal{B},
\end{equation}
 is contained in the vector space $\varphi \mathcal{B}$.
\end{lemma}

\begin{proof}
  For $b(t) = 0$, the thesis of the lemma  is a direct consequence of
  Proposition~\ref{prop:k=-2}. Hence, we assume that $b(t)\neq 0$, and
  we rewrite equation~\eqref{eq:hv-2} as a system 
\begin{equation}
\label{eq:sysk-2}
 \begin{bmatrix}
       \dot{x}\\
       \dot{y}
     \end{bmatrix} = \begin{bmatrix}
       0 & 1\\
       - \lambda \varphi^{-4} & 0
     \end{bmatrix}\begin{bmatrix}
       x\\
      y
     \end{bmatrix} + \begin{bmatrix}
       0\\
       b\varphi^{-3}
     \end{bmatrix}, \qquad b=b(t) . 
\end{equation}
A particular solution of this system is given by 
\begin{equation*}
 \begin{bmatrix}
       x\\
       y
     \end{bmatrix}=\vX \begin{bmatrix}
       c_1\\
       c_2
     \end{bmatrix},
\end{equation*}
where $\vX$ is a fundamental matrix of the homogeneous part of
equation~\eqref{eq:hv-2}, and $[c_1, c_2]^T$ is a certain solution of
the following equation 
\begin{equation*}
 \begin{bmatrix}
      \dot c_1\\
      \dot  c_2
     \end{bmatrix}= \vX^{-1}\begin{bmatrix}
       0\\
       b\varphi^{-3}
     \end{bmatrix}.
\end{equation*}

 Let us write $\lambda = 1 - 4 \omega^2$. In the case
 $\omega=0$,  by Proposition~\ref{prop:k=-2}, a
  fundamental matrix $\vX$ and its inverse are  given by 
\begin{equation*}
\vX= \begin{bmatrix} \varphi & \varphi I\\
       \dot{\varphi} & \dot\varphi I+ \varphi \dot I
     \end{bmatrix}, \qquad 
\vX^{-1}= \frac{1}{2}\begin{bmatrix} 
\dot\varphi I+ \varphi \dot I  & -\varphi I\\
       -\dot{\varphi} & \varphi 
     \end{bmatrix}.
\end{equation*}
Hence, we have 
\begin{equation*}
 \begin{bmatrix}
      \dot c_1\\
      \dot  c_2
     \end{bmatrix}= \frac{1}{2}
\begin{bmatrix}
       bI\varphi^{-2}\\
       b\varphi^{-2}
     \end{bmatrix},
\end{equation*}
so, by Lemma~\ref{lem:int}, $[ c_1, c_2]\in\mathcal{B}\times \mathcal{B}$.
Hence, $[ x, y]\in \varphi\mathcal{B}\times \varphi\mathcal{B}$
and this finishes the proof for $\omega=0$. 

For $\omega\neq0$, the fundamental matrix $\vX$ and its inverse for
homogeneous part of equation~\eqref{eq:sysk-2}, are given by 
\begin{equation*}
\vX= \begin{bmatrix} \varphi E_{\omega} & \varphi E_{-\omega} \\
       \dot{\varphi} E_{\omega}+\varphi {\dot E}_{\omega} &
       \dot\varphi E_{-\omega}+ \varphi {\dot E}_{-\omega}
     \end{bmatrix}, \qquad 
\vX^{-1}= \frac{1}{2\omega}\begin{bmatrix} 
 \dot\varphi E_{-\omega}+ \varphi {\dot E}_{-\omega} &-\varphi E_{-\omega}\\
       - \dot{\varphi} E_{\omega}\varphi {\dot E}_{\omega}  & \varphi E_{\omega} 
     \end{bmatrix},
\end{equation*}
so 
\begin{equation}
\label{eq:cexo}
 \begin{bmatrix}
      \dot c_1\\%%
      \dot  c_2
     \end{bmatrix}= \frac{1}{2\omega}
\begin{bmatrix}
       b\varphi^{-2} E_{-\omega}\\
       b\varphi^{-2} E_{\omega}
     \end{bmatrix}.
\end{equation}
Hence, also in this case, by Lemma~\ref{lem:int}, $[ c_1, c_2]\in\mathcal{B}\times \mathcal{B}$.
Hence, $[ x, y]\in \varphi\mathcal{B}\times \varphi\mathcal{B}$
and this finishes the proof. 

\end{proof}

\subsection{Proof of Theorem~\ref{thm:k=-2}}

\paragraph{The first step.}
With the use of vectorial notations, we show the following result.   
\begin{proposition}
For \VE{1} we have  $\mathrm{Sol} (\mathrm{VE}_1) \subset \varphi
\mathcal{B}^n$.
\end{proposition}
\begin{proof}
The first order variational equation has the form 
\begin{equation}
\label{eq:ve1k-2}
\ddot \vq_1= -\varphi^{-4}V''(\vd)\vq_1. 
\end{equation}
We consider two possibilities: either  $V''(\vd)$ is  
diagonalisable, or it is not  diagonalisable.

If $V''(\vd)$ is   diagonalisable, then $\mathrm{VE}_1$ splits into
a direct sum of  equations which have the form 
\begin{equation}
\label{eq:ve1dir}
\ddot x= - \frac{\lambda }{\varphi^4} x, 
\end{equation}
where $\lambda$ is an eigenvalue of $V''(\vd) $. And the result follows from Lemma~\ref{lem:k=-2}.

If $V'' (d)$ is not diagonalisable, $\mathrm{VE}_1$ splits into subsystems
corresponding to Jordan block of the form
\begin{equation}
\label{eq:jor}
\left\{ \begin{array}{lll}
     \ddot{x} & = & - \lambda\varphi^{-4}x,\\
     \ddot{y} & = & - \lambda\varphi^{-4} y + \varphi^{-4}x, \\
     \ddot{z} & = & - \lambda\varphi^{-4} z + \varphi^{-4}y, \\
     \cdots & = & \cdots,\\
      \ddot{w}&=&  - \lambda\varphi^{-4} w.
   \end{array} \right.
\end{equation}

By Lemma~\ref{lem:k=-2}, an arbitrary solution $x$ of the first
equation can be written in  the form $x= \varphi b_{0}$, where
$b_{0}\in\mathcal{B}$. Hence, for this solution,  the second equation
in~\eqref{eq:jor} 
can be written in the form
\begin{equation*}
\ddot{y} =
-\frac{\lambda}{\varphi^4} y + \frac{b_0}{\varphi^3}.
\end{equation*}
Thus, by Lemma~\ref{lem:k=-2}, $y \in \varphi \mathcal{B}$. Inserting this
solution  into the third
equation~\eqref{eq:jor} we also get that $z \in \varphi \mathcal{B}$. 
If the size of the Jordan block is greater than three we proceed 
inductively.  As a
consequence we obtain that an arbitrary solution $\vq_1$ belongs to 
$\varphi \mathcal{B}^n$, and this finishes our proof.
\end{proof}
\paragraph{The induction procedure.}

Let us write the higher variational equation in Newton form, that is again, we set $\vF
(\vq) := - V' (q)$.  Here for  convenience, we  write explicitly  the first
three variational equations 
\[ \left\{ \begin{array}{lll}
     \ddot{\vq}_1 & = \dfrac{\vF' (\vd) (\vq_1)}{\varphi^4},  & \\
     \ddot{\vq}_2 & = \dfrac{\vF' (\vd) (q_2)}{\varphi^4} + \dfrac{F'' (\vd) (\vq_1,
     \vq_1)}{\varphi^5},  & \\
     \ddot{\vq}_3 & = \dfrac{F' (\vd) (\vq_1)}{\varphi^4} + 3 \dfrac{\vF'' (\vd) (\vq_1,
     \vq_2)}{\varphi^5} + \dfrac{\vF^{(3)} (\vd) (\vq_1, \vq_1, \vq_1)}{\varphi^6}. & 
   \end{array} \right. \]
As a consequence, in general,  we may write the $\mathrm{VE}_p$ in the form
\[ \ddot{\vq}_p = \frac{\vF' (\vd) (\vq_p)}{\varphi^4} + \frac{\vR_p}{\varphi^3}, \]
where $R_p$ is a $\C$-linear combination of terms of  the form
\[ \frac{\vF'' (\vd) (\vq_{i_1}, \vq_{i_2})}{\varphi^2}, \cdots, \frac{\vF^{(s)} (\vd)
   (\vq_{i_1}, \ldots, \vq_{i_s})}{\varphi^s}, \]
with $2 \leq s \leq p$, and $1 \leq i_k \leq p - 1$.

Now, let assume that by solving the $\mathrm{VE}_m$ for $1 \leq m
\leq p - 1$ we found that
\[ \vq_m \in \varphi \mathcal{B}^n . \]
In other words $\vq_m = \varphi \vB_m$ with some $\vB_m \in \mathcal{B}^n$. Since
$\vF^{(s)} (\vd)$ is a vectorial $s$-form with constant coefficients we  get
\[ \frac{\vF^{(s)} (\vd) (\vq_{i_1}, \ldots, \vq_{i_s})}{\varphi^s} = \frac{\vF^{(s)}
   (\vd) (\varphi \vB_{i_1}, \ldots, \varphi \vB_{i_s})}{\varphi^s} = F^{(s)} (d)
   (\vB_{i_1}, \ldots, \vB_{i_s}) . \]
But now, since $\mathcal{B}$ is stable by multiplication, it follows that
\[ \vR_p \in \mathcal{B}^n . \]
According to Lemma~\ref{lem:k=-2}, we can deduce that all the solutions $\vq_p$ of
the equations
\[ \ddot{\vq}_p = \frac{\vF' (\vd) (\vq_p)}{\varphi^4} + \frac{\vR_p}{\varphi^3}, \]
belong to $\varphi \mathcal{B}^n$.

This ends the proof of the theorem.

\appendix
\section{About the analytic behaviour of some trigonometric integrals}
We show the following.
\begin{lemma}
\label{lem:a1}
Let $f\in\scO(\C)$ be a holomorphic function such that $f(n\pi)\neq0$,
for a certain $n\in\Z$. Then, the function 
\begin{equation}
\label{eq:f}
t\longmapsto\ F(t):=\int \frac{f(t)}{\sin t}\, \rmd t,
\end{equation}
is not meromorphic on $\C$. 
\end{lemma}
Indeed, if $\cM_{n\pi}$ denotes the monodromy operator around $t=n\pi$, we get
\begin{equation*}
\cM_{n\pi}(F) =F +2\pi \rmi f(n\pi).
\end{equation*}

For $\omega\in\C$ and $n\in\N$, we  define the following functions $T_{n}^{(\omega)}$
\begin{equation}
\label{eq:fon}
T^{(\omega)}_{n}(t):=\int \frac{\rme^{\rmi \omega
    t}}{\sin^nt}\,\rmd t.
\end{equation}
\begin{lemma}
\label{lem:a2}
The following statements holds true.
\begin{enumerate}
\item The function $T_{1}^{(\omega)}(t)$ is not meromorphic. 
\item  The function $T_{2}^{(\omega)}(t)$ is meromorphic iff $\omega=0$.
\end{enumerate}
\end{lemma}
\begin{proof}
The first statement immediately follows from Lemma~\ref{lem:a1}. In
order to prove the second one, we notice that in a neighborhood
of $t=0$ we have
\begin{equation*}
\sin^2t =t^2(1+\cdots), \mtext{and} \rme^{\rmi \omega t}=1+\rmi \omega
t +\cdots,
\end{equation*}
and thus
\begin{equation*}
\frac{\rme^{\rmi\omega t}}{\sin^2t}=\frac{1}{t^2}+\frac{\rmi \omega}{t}+\cdots.
\end{equation*}
Hence if $\omega\neq0$ the function $T_{2}^{(\omega)}(t)$ is not
meromorphic. For $\omega=0$, we have  $T_{2}^{(0)}(t)=-\cot t$,
is a meromorphic function on $\C$.
\end{proof}
\begin{proposition}
For $n>2$ the following relation holds true
\begin{equation}
\label{eq:recf}
T_{n}^{(\omega)}(t)=
 g_{n-2}^{(\omega)}(t)+c_{n-2}^{(\omega)} T_{n-2}^{(\omega)}(t),
\end{equation}
where
\begin{equation*}
 g_{n-2}^{(\omega)}(t):=-\rme^{\rmi\omega t}\frac{\rmi\omega\sin t
   +(n-2)\cos t}{(n-1)(n-2)\sin^{n-1}t}, \mtext{and}
 c_{n-2}^{(\omega)}:= \frac{(n-2)^2-\omega^2}{(n-1)(n-2)}.
\end{equation*}
\end{proposition}
\begin{proof}
Differentiating both sides of~\eqref{eq:recf} and making some
simplifications we obtain the identity.
\end{proof}
Now, we are ready to characterize all the cases when 
$T_{n}^{(\omega)}$ is meromorphic.
\begin{lemma}
The function
$T_{n}^{(\omega)}$ is meromorphic if and only if we either get
\begin{enumerate}
\item\label{item:1} $\omega\in\defset{\pm 2k}{0\leq k
    \leq (n-2)/2}$ for even $n$ and $n\geq 2$,
\item\label{item:2}$\omega\in\defset{\pm (1+2k)}{0\leq k
    \leq (n-1)/2}$ for odd $n$ and $n \geq 3$.
\end{enumerate}
\end{lemma}
\begin{proof}
Assume that $n$ is even. Then from relation~\eqref{eq:recf} it follows
that
\begin{equation*}
T_{n}^{(\omega)}(t) = f_{n}^{(\omega)}(t)+p_{(n)}(\omega) T_{2}^{(\omega)}(t),
\end{equation*}
where $f_{n}^{(\omega)}$ is a meromorphic function,  and
\begin{equation*}
p_{n}(\omega)=\prod_{k=0}^{(n-2)/2}c_{2k}^{(\omega)}=a_n
\prod_{k=0}^{(n-2)/2}\left[(2k)^2-\omega^2\right], \qquad a_n\neq0.
\end{equation*}
Thus, by Lemma~\ref{lem:a2}, $T_{n}^{(\omega)}$ is meromorphic
iff $p_{n}(\omega)=0$, and this occurs iff $\omega$ belongs to the
set given in the first statement. 

With the case of odd $n$ we proceed in a similar way. 
\end{proof}

\newcommand{\noopsort}[1]{}\def\cprime{$'$} \def\cprime{$'$}
  \def\polhk#1{\setbox0=\hbox{#1}{\ooalign{\hidewidth
  \lower1.5ex\hbox{`}\hidewidth\crcr\unhbox0}}} \def\cprime{$'$}
  \def\cydot{\leavevmode\raise.4ex\hbox{.}} \def\cprime{$'$} \def\cprime{$'$}
  \def\cprime{$'$} \def\cprime{$'$} \def\cprime{$'$}

%\bibliographystyle{klunamed}
%\bibliography{mathreva,ajm,audin,morales,books,ziglin,churchill,dgt,mp,oldies,%

\begin{thebibliography}{}

\bibitem[\protect\citeauthoryear{Audin}{2001}]{Audin:01::c}
Audin, M.: 2001, {\em Les syst\`emes hamiltoniens et leur int\'egrabilit\'e},
  Cours Sp\'ecialis\'es 8, Collection SMF.
\newblock Paris: SMF et EDP Sciences.

\bibitem[\protect\citeauthoryear{Baider et~al.}{1996}]{Churchill:96::b}
Baider, A., R.~C. Churchill, D.~L. Rod, and M.~F. Singer: 1996, `On the
  Infinitesimal Geometry of Integrable Systems'.
\newblock In: {\em Mechanics Day (Waterloo, ON, 1992)}, Vol.~7 of {\em Fields
  Inst. Commun.}
\newblock Providence, RI: Amer. Math. Soc., pp. 5--56.

\bibitem[\protect\citeauthoryear{Casale}{2009}]{Guy09}
Casale, G.: 2009, `Morales-{R}amis Theorems via {M}algrange pseudogroup'.
\newblock {\em Annales de l'Institut Fourier} {\bf 59}(7), 2593--2610.

\bibitem[\protect\citeauthoryear{Duval and
  Maciejewski}{2009}]{Maciejewski:09::}
Duval, G. and A.~J. Maciejewski: 2009, `Jordan obstruction to the integrability
  of {H}amiltonian systems with homogeneous potentials'.
\newblock {\em Annales de l'Institut Fourier} {\bf 59}(7), 2839--2890.

\bibitem[\protect\citeauthoryear{Grigorenko}{1975}]{MR0376633}
Grigorenko, N.~V.: 1975, `Abelian extensions in {P}icard-{V}essiot theory'.
\newblock {\em Mat. Zametki} {\bf 17}, 113--117.

\bibitem[\protect\citeauthoryear{Humphreys}{1975}]{Humphreys:75::}
Humphreys, J.~E.: 1975, {\em Linear algebraic groups}.
\newblock New York: Springer-Verlag.
\newblock Graduate Texts in Mathematics, No. 21.

\bibitem[\protect\citeauthoryear{Kolchin}{1968}]{Kolchin:68::}
Kolchin, E.~R.: 1968, `Algebraic groups and algebraic dependence'.
\newblock {\em Amer. J. Math.} {\bf 90}, 1151--1164.

\bibitem[\protect\citeauthoryear{Maciejewski and
  Przybylska}{2009}]{Maciejewski:09::b}
Maciejewski, A.~J. and M. Przybylska: 2009, `Differential {G}alois theory and
  integrability'.
\newblock {\em Internat. J. Geom. Methods in Modern Phys.} {\bf 6}(8),
  1357--1390.

\bibitem[\protect\citeauthoryear{Morales-Ruiz and Ramis}{2010}]{Morales:10::a}
Morales-Ruiz, J.~J. and J.-P. Ramis: 2010, `Integrability of dynamical systems
  through differential {G}alois theory: a practical guide'.
\newblock In: {\em Differential algebra, complex analysis and orthogonal
  polynomials}, Vol. 509 of {\em Contemp. Math.}
\newblock Providence, RI: Amer. Math. Soc., pp. 143--220.

\bibitem[\protect\citeauthoryear{Morales-Ruiz et~al.}{2006}]{Morales:07::}
Morales-Ruiz, J.~J., J.~P. Ramis, and C. Sim{\'o}: 2006, `Integrability of
  {H}amiltonian systems and differential {G}alois groups of higher variational
  equations'.
\newblock {\em Ann. Sci. \'Ec. Norm. Sup\'er} {\bf 40}(6), 845--884.

\bibitem[\protect\citeauthoryear{van~der Put and Singer}{2003}]{MR1960772}
van~der Put, M. and M.~F. Singer: 2003, {\em Galois theory of linear
  differential equations}, Vol. 328 of {\em Grundlehren der Mathematischen
  Wissenschaften [Fundamental Principles of Mathematical Sciences]}.
\newblock Berlin: Springer-Verlag.

\end{thebibliography}
%grigorenko,singer}
\end{document}